\documentclass[a4paper]{article}
\usepackage{amsmath, amssymb, amsthm, amsfonts, mathtools, mathrsfs, eufrak}
\mathtoolsset{showonlyrefs} 
\newcommand{\si}{\sigma}
\newcommand{\la}{\lambda}

\newcommand{\pa}{\partial}

\newcommand{\al}{\alpha}
\newcommand{\be}{\beta}
\newcommand{\de}{\delta}

\newcommand{\ga}{\gamma}
\newcommand{\ka}{\kappa}

\newcommand{\om}{\omega}
\newcommand{\ve}{\varepsilon}
\newcommand{\ze}{\zeta}

\newcommand{\cd}{\cdot}

\newcommand{\cs}{\operatornamewithlimits{cs}}
\newcommand{\R}{{\mathbb R}}

\newcommand{\Z}{{\mathbb Z}}
\newcommand{\N}{{\mathbb N}}

\newcommand{\cE}{{\cal E}}

\newcommand{\tm}{\widetilde{m}}

\newcommand{\cT}{{\cal{T}}}

\renewcommand{\(}{\left(}
\renewcommand{\)}{\right)}

\renewcommand{\th}{\theta}

\newtheorem{theorem}{\bf Theorem}[section]

\newtheorem{lemma}[theorem]{\bf Lemma}
\newtheorem{proposition}[theorem]{\bf Proposition}

\theoremstyle{remark}
  \newtheorem{remark}[theorem]{\sc Remark}
\theoremstyle{definition}
  
  \newtheorem{example}[theorem]{\sc Example}
\numberwithin{equation}{section}

\begin{document}

\title{
On the energy decay estimates for the weak dissipative wave equations with oscillating coefficient}
\author{
Fumihiko Hirosawa\!
\footnote{
Department of Mathematical Sciences, Faculty of Science, Yamaguchi University, Japan; 
e-mail: hirosawa@yamaguchi-u.ac.jp}
\;\; and \;\;
Daichi Nakajima\!
\footnote{Graduate School of Sciences and Technology for Innovation, Yamaguchi University, Japan; 
e-mail: d008vbu@y-u.jp}
}

\date{}

\maketitle

\begin{abstract}
It is known that the asymptotic behavior of time-dependent dissipative coefficient in the Cauchy problem of dissipative wave equation dominates the energy decay estimate. 
In particular, it is important to study the case where the dissipative coefficient behave like $1/t$ as $t$ goes to infinity, which is called weak dissipation, because its order is close to the critical case of decay and non-decay. 
In this case, an oscillating perturbation of weak dissipation can give a crucial effect on the energy decay estimate, but the analysis is very difficult compared to the case without oscillations. 
In this paper, we develop a method recently introduce by Ghisi-Gobbino that has contributed to a precise analysis for dissipative wave equations with oscillating weak dissipation, and consider the effect of the oscillations, which is more general and close to the critical case. 
Furthermore, we study the effect of the smoothness of the initial data on the energy decay estimates. 
\end{abstract}

\bigskip

\begin{center}
\end{center}

\bigskip
\begin{center}
\textbf{Mathematics Subject Classification 2020 (MSC2020):} 
35L15, 35L05, 35B40
\end{center}

\bigskip

%
%
\section{Introduction}
%
%

%
%

Let us consider the following Cauchy problem for dissipative wave equation: 
\begin{equation}\label{u}
\begin{dcases}
  \(\pa_t^2 - \Delta + b \pa_t\) u(t,x) = 0, &
  (t,x) \in (0,\infty) \times \R^d,\\
  u(0,x)=u_0(x),\;\; \pa_t u(0,x) = u_1(x), 
  & x \in \R^d, 
\end{dcases}
\end{equation}
where $d \in \N$, 
$b=b(t)$ is real valued and $\Delta=\sum_{j=1}^d \pa_{x_j}^2$. 
Then, the total energy of the solution to \eqref{u} at $t$ is defined by 
\begin{equation}
  E(t)=E(t;u):= \|\nabla u(t,\cd)\|^2 + \|\pa_t u(t,\cd)\|^2, 
\end{equation}
where $\nabla = {}^t(\pa_{x_1},\ldots,\pa_{x_d})$ and 
$\|\;\;\|$ denotes the usual $L^2$ norm in $\R^d$. 
Since our interest is the energy estimates for linear Cauchy problem 
\eqref{u}, the existence of a solution is assumed hereafter. 

The following equality is immediately obtained by differentiating $E(t;u)$ with respect to $t$: 
\begin{equation}\label{dE}
  \frac{d}{dt}E(t;u) = -2b(t)\|\pa_t u(t,\cd)\|^2. 
\end{equation}
\eqref{dE} implies that $E(t;u)$ is monotone decreasing if $b(t)>0$, 
moreover, $\int^\infty_0 b(t)\,dt = \infty$ must be hold so that the energy decays as $t\to\infty$, that is, $\lim_{t\to\infty}E(t;u)=0$. 
Therefore, it is natural question what order $E(t;u)$ decays as $t\to\infty$ for $b$ that satisfies $b(t) \ge m/(1+t)$ with $m>0$. 

If $b$ is a positive constant, then the following energy decay estimate holds:
\begin{equation}\label{Eest0}
  E(t;u) \le C (1+t)^{-\tm}\(E(0;u) + \|u_0(\cd)\|^2\)
\end{equation}
for $\tm =1$, where $C$ is a positive constant depends only on $b$ (\cite{M76}). 
Hereafter, we will represent it by $C$ if a positive constant exists. 
If $b$ is given by 
\begin{equation}\label{b0}
  b(t)=\frac{m}{1+t} 
\end{equation}
for $m>0$, which is called weak dissipation, \eqref{Eest0} holds for 
\begin{equation}\label{tm0}
  \tm = \min\{2,m\}. 
\end{equation}
The problem of whether \eqref{Eest0} holds for what $\tm$ in the intermediate cases between constant dissipation and weak dissipation has been studied in \cite{W07} and elsewhere, and these are summarized in \cite{GG25}. 
In this paper, we are particularly interested in the case where $b(t)$ behaves like \eqref{b0}, we shall also call it weak dissipation, so let us explain a bit more detail below. 

If $b$ is given by \eqref{b0} with $m=2$, then we have the following decay estimate: 
\begin{equation}\label{Eest_m=2}
  \|\nabla  u(t,\cd)\|^2 
+ \left\|\pa_t u(t,\cd) + (1+t)^{-1}u(t,\cd)\right\|^2
 =(1+t)^{-2} \(\|\nabla  u_0(\cd)\|^2 + \|u_1(\cd) + u_0(\cd)\|^2\), 
\end{equation}
that is derived from the energy conservation 
$E(t;w)=E(0;w)$ for $w:=(1+t)u$. 
Here we notice that not only $(u_0,u_1) \in \dot{H}^1 \times L^2$, that is $E(0;u)<\infty$, but also $(u_0,u_1) \in H^1 \times L^2$ seems to be required for the decay estimate \eqref{Eest0}, where $H^s$ and $\dot{H}^s$ denote the usual Sobolev and homogeneous Sobolev spaces of order $s \ge 0$ in $\R^d$, respectively. 
For more general $b \in C^1([0,\infty))$, it was proved in \cite{M77} that 
\eqref{Eest0} also holds for any $\tm$ that satisfies both the following conditions: 
\begin{equation}\label{bb'}
  b(t) \ge \tm(1+t)^{-1}
  \;\;\text{ \textit{and} }\;\;
  (\tm-1)(\tm-2)-(\tm-1)(1+t)b(t)-(1+t)^2 b'(t) \ge 0. 
\end{equation}
Indeed, \eqref{bb'} provides \eqref{tm0} if $b$ is given by \eqref{b0}. 
Actually, \eqref{bb'} is also applicable if $b(t)$ is oscillating, however, the influence of non-constant $b(t)$ is crucial in determining of $\tm$. 
Since we are interested in the case that $b(t)$ has oscillations and to focus on this, we suppose that $b(t)$ is represented by
\begin{equation}\label{de(t)}
  b(t) = \frac{m}{1+t} + \de(t), 
\end{equation}
where $\de(t)$ denotes an oscillatory perturbation, and we shall derive the conditions for $\de(t)$ so that the same energy decay estimate \eqref{Eest0} holds as when $\de(t)=0$. 

Here we introduce two results that directly motivated the main theorems of this paper. 

\begin{theorem}[\cite{GH25,HW08}]\label{ThmGH}
Let $0<m<1$ and $\de \in C^k([0,\infty))$ with $k \ge 1$. 
If there exist $\be \in \R$ and $\ga \ge 0$ satisfying 
\begin{equation}\label{ThmGH-bega}
  \be \ge \frac{1-k\ga}{k+1}
\end{equation}
such that 
\begin{equation}\label{ThmGH-e1}
  \max_{0\le l \le k}
  \sup_{t \ge 0}\left\{
  (1+t)^{\be(l+1)} \left|\de^{(l)}(t)\right|
  \right\}<\infty
\end{equation}
and
\begin{equation}\label{ThmGH-e2}
  \sup_{t \ge 0}\left\{
  (1+t)^{\ga}
  \int^\infty_t\left|\int^\infty_s \de(\tau)\,d\tau\right|\,ds
  \right\}<\infty, 
\end{equation}
then \eqref{Eest0} holds for $\tm=m$. 
\end{theorem}

\begin{remark}
The conditions \eqref{ThmGH-e1} and \eqref{ThmGH-e2} are called the $C^k$-property and the stabilization condition, respectively. 
\end{remark}

\begin{theorem}[\cite{GG25}]\label{ThmGG}
If $\de(t)$ satisfies either the following (i) or (ii): 
\begin{itemize}
\item[(i)] 
$\de \in L^1_{loc}([0,\infty))$, 
$\de(t) \ge 0$ and $\sup_{t\ge 0}\{(1+t)\de(t)\}<\infty$. 
\item[(ii)] 
For $\al>1$ and $0<r \le m$, $\de(t)$ is given by 
\begin{equation}\label{ex_de-GG}
  \de(t) = \frac{r}{1+t}\sin\((1+t)^\al\), 
\end{equation}
\end{itemize}
then \eqref{Eest0} with \eqref{tm0} holds. 
On the other hand, there exists $\de(t)$ satisfies (iii): 
\begin{itemize}
\item[(iii)] 
$\de \in C^\infty([0,\infty))$, 
$\sup_{t\ge 0}\{(1+t)|\de(t)|\}=:r \le m$ 
and
$\int^\infty_0\de(t)\,dt$ exists, 
\end{itemize}
such that \eqref{Eest0} with \eqref{tm0} does not hold. 
More precisely, the following estimate holds: 
\begin{equation}
  \inf_{t\ge 0}\left\{(1+t)^{m-r/2}E(t;u)\right\}>0 
\end{equation}
since $E(0;u)>0$. 
\end{theorem}

\begin{remark}\label{Rem_ThmGH}
Theorem \ref{ThmGH} can be applied to \eqref{ex_de-GG} only for $m<1$ and $\al>1$ without any assumption to $r$ such as in Theorem \ref{ThmGG} (ii). 
Indeed, \eqref{ThmGH-e1} and \eqref{ThmGH-e2} hold for 
$\be = -\al+1 + \al/(k+1)$ and $\ga = \al-1$, 
and thus \eqref{ThmGH-bega} is valid (see \cite{GH25}). 
\end{remark}

%
%
\section{Main theorems}
%
%

From now on, we particularly consider the following dissipative coefficient $b(t)$: 
\begin{equation}\label{b=si_0}
  b(t) = \frac{\si_0(\eta(t))}{1+t},
\end{equation}
where $\si_0(\eta)$ $(\eta \in \R)$ and $\eta(t)$ $(t \ge 0)$ satisfy the followings: 
\begin{itemize}
\item[(A1)] 
$\si_0(\eta)$ is periodic and H\"older continuous of order $\al>1/2$. 
\item[(A2)] 
$\si_0(\eta)$ satisfies either of the following conditions: 
\begin{equation}\label{b>=0}
  \min_{\eta \in \R}\{\si_0(\eta)\} \ge 0
\end{equation}
or
\begin{equation}\label{0<m<1}
  m:=\lim_{\eta\to\infty} \frac{1}{\eta}\int^\eta_0\si_0(s)\,ds \in (0, 1). 
\end{equation}
The conditions \eqref{b>=0} and \eqref{0<m<1} are called 
\textit{non-negative dissipative condition} 
and
\textit{non-effective dissipative condition} respectively. 
\item[(A3)] 
$\eta \in C^2([0,\infty))$, $\eta(0) \ge 0$, $\eta'(t) > 0$ 
and $\eta''(t)>0$ for any $t \ge 0$. 
\end{itemize}
%

Let the period of $\si_0$ be $2T$ and denote the Fourier series of $\si_0$ by 
\begin{equation}\label{Fourier_si0}
  \si_0(\eta)
 =m+\sum_{n=1}^\infty 
  \(\al_n \cos\(\frac{n\pi\eta}{T}\) + \be_n \sin\(\frac{n\pi\eta}{T}\)\). 
\end{equation}
Then we set $m$ and $\de(t)$ of \eqref{de(t)} by 
\begin{equation}\label{de=si}
  m = \frac{1}{2T}\int^T_{-T} \si_0(s)\,ds
  \;\;\text{ \textit{and} }\;\;
  \de(t)=\frac{\si(\eta(t))}{1+t}, 
\end{equation}
where 
\begin{equation}\label{si(eta)}
  \si(\eta):=\sum_{n=1}^\infty 
  \(\al_n \cos\(\frac{n\pi\eta}{T}\) + \be_n \sin\(\frac{n\pi\eta}{T}\)\), 
\end{equation}
that is, 
\begin{equation}\label{b=si}
  b(t) = \frac{m}{1+t} + \frac{\si(\eta(t))}{1+t}. 
\end{equation}
Here we note that $m$ in \eqref{0<m<1} and \eqref{de=si} are the same. 
Moreover, the following estimates hold: 
\begin{equation}\label{FC}
  A_1:= \sum_{n=1}^\infty \(|\al_n|+|\be_n|\) <\infty 
\end{equation}
since (A1) is valid (see \cite{Kz}), and
\begin{equation}\label{B0}
  B_0:=
  \sup_{\tau_-,\tau_+ \in [0,\infty)}
  \left\{\left|
  \int^{\tau_+}_{\tau_-} \frac{\si(\eta(s))}{1+s}\,ds\right|\right\}
  <\infty
\end{equation}
by Lemma \ref{lemma_int_si0}. 

Under the assumption (A3), we define $\mu(t)$ on $[0,\infty)$ by 
\begin{equation}
  \mu(t):=\max_{0 \le \tau \le t}
  \left\{\frac{1}{(1+\tau)\eta''(\tau)}\right\}. 
\end{equation}

\begin{remark}
If (A3) holds, then $\inf_{t\ge 0} \{\eta(t)/(1+t)\}>0$. 
In particular, if $\sup_{t \ge 0}\{\mu(t)\}<\infty$, then 
$\lim_{t\to\infty} \eta(t)/(1+t) =\infty$, 
that is, $\si(\eta(t))$ oscillates faster than the periodic function $\si(1+t)$. 
Indeed, denoting $\sup_{t\ge 0}\{\mu(t)\} = \mu_1$, 
and noting that 
\begin{align*}
  \int^t_0 \int^s_0 \eta''(\tau)\,d\tau\,ds
\begin{dcases}
=\eta(t)-\eta(0)-\eta'(0)t,
  \\
\ge 
  \int^t_0 \int^s_0 \frac{1}{\mu_1 (1+\tau)} \,d\tau\,ds
 =\frac{(1+t)\log(1+t) - t}{\mu_1},
\end{dcases}
\end{align*}
we have 
\begin{align*}
  \frac{\eta(t)}{1+t}
  \ge
  \frac{(1+t)\log(1+t) - t + \mu_1(\eta(0)+\eta'(0)t)}{\mu_1(1+t)}
  \to \infty
  \;\;(t\to\infty). 
\end{align*}
\end{remark}

In the above setting, our first theorem is given as follows. 
\begin{theorem}\label{Thm1}
Let (A1), (A2) and (A3) be valid. 
For any $\ve>0$ and $\tm$ satisfying 
\begin{equation}\label{Thm1_tm}
  0 < \tm \le \tm_0:=
\begin{cases}
  \min\{2,m\} & \;\text{ \textit{for} }\; \eqref{b>=0},\\ 
  m & \;\text{ \textit{for} }\; \eqref{0<m<1},
\end{cases}
\end{equation}
there exists a positive constant $C$ such that the following estimate holds: 
\begin{equation}\label{Eest-Thm}
  E(t) \le 
  C \(1+\frac{\exp\(\ve \mu(t)\)}{(1+t)^{m-\tm}}\)
  (1+t)^{-\tm}
  \(E(0) + \|u_0(\cd)\|_{\dot{H}^{1-\tm/2}}^2\). 
\end{equation}
\end{theorem}

\begin{remark}\label{rem-IC-H}
If $\sup_{t \ge 0}\{\mu(t)\}<\infty$ holds as in the examples below, then \eqref{Eest-Thm} gives an energy decay estimate of the same order as \eqref{Eest0}. 
In particular, if $\tm < 2$, then less assumption is required to $u_0$ since 
$H^1 \subsetneq \dot{H}^1 \cap \dot{H}^{1-\tm/2}$. 
From a different point of view, if the assumption $u_0 \in \dot{H}^1 \cap \dot{H}^{\ka}$ weakens as $\ka$ becomes larger to $1$, then the decay rate is smaller, and finally no decay at $\ka=1$. 
\end{remark}

\begin{remark}
$L^p$-$L^q$ decay estimates to the solution of \eqref{u} with $b\equiv 1$ in homogeneous Sobolev space is studied in \cite{IIOW19}. 
Here we note that the equation of \eqref{u} is more parabolic-like if $b(t)$ is a positive constant, and the influence of the initial data is slightly different from our model, which is closer to a free wave equation.  
\end{remark}

\begin{example}\label{ex_eta=t}
Let $\eta(t) = (1+t)^\al$ and $\al>0$. 
Since 
\[
  (1+t)\eta''(t) = \al(\al-1)(1+t)^{\al-1}, 
\]
(A3) is valid for $\al>1$, and then 
\[
  \mu(t) 
 =\frac{1}{\al(\al-1)}. 
\]
Here we note that Theorem \ref{ThmGG} (ii) corresponds to the case 
$\si_0(\eta)=m+r\sin(\eta)$ with $0< r \le m$. 
\end{example}

The next examples consider limit cases of Example \ref{ex_eta=t} as $\al \to 1+0$. 

\begin{example}\label{ex_eta=log}
Let $\eta(t) = (e+t)(\log(e+t))^\be$ and $\be>0$. 
Since 
\begin{align*}
  \eta''(t)
 =\frac{\be}{e+t}
  \(\log(e+t)\)^{\be-1}\(1+(\be-1)\(\log(e+t)\)^{-1}\), 
\end{align*}
(A3) is valid, and then 
\begin{align*}
  \mu(t)
  \begin{cases}
    = e\be^{-2} & (\be \ge 1), \\
    \le e\be^{-2}\(\log(e+t)\)^{1-\be} & (\be<1). 
  \end{cases}
\end{align*}
Here we note that 
$\lim_{t\to\infty}\exp((\log(e+t))^{1-\be})/t^\nu=0$ 
for any $0<\be<1$ and $\nu>0$. 
Therefore, for any $\tm<m$, there exists a positive constant $C$ such that the following estimate holds 
\begin{equation}\label{Eest-ex_eta=log}
  E(t) \le C (1+t)^{-\tm}
  \(E(0) + \|u_0(\cd)\|_{\dot{H}^{1-\tm/2}}^2\). 
\end{equation}
\end{example}

\begin{example}\label{ex_eta=loglog}
Let $\eta(t) = (e^e+t)(\log\log(e^e+t))^\ga$ and $\ga>0$. 
Since 
\begin{align*}
  \eta''(t)
&=\frac{\ga \(\log\log\(e^e+t\)\)^{\ga-1}}{\(e^e+t\)\log\(e^e+t\)}
  \(1-\frac{1}{\log\(e^e+t\)}
     +\frac{\ga-1}{\log\(e^e+t\)\log\log\(e^e+t\)}\)
\\
&\ge
  \frac{\ga \(\log\log\(e^e+t\)\)^{\ga-1}}{\(e^e+t\)\log\(e^e+t\)}
  \(1-\frac{2}{e}\)
\\
&\ge
  \frac{\ga(e-2)}{e^{e+1}}
  \frac{\(\log\log\(e^e+t\)\)^{\ga-1}}{(1+t)\log\(e^e+t\)},
\end{align*}
(A3) is valid, and then 
\begin{align*}
  \mu(t)
   \le 
   \frac{e^{e+1}}{\ga(e-2)}
   \frac{\log\(e^e+t\)}{\(\log\log\(e^e+t\)\)^{\ga-1}}. 
\end{align*}
If $\ga \ge 1$, then $\lim_{t\to\infty} \mu(t)/\log t < \infty$. 
Therefore, for any $\tm < m$, \eqref{Eest-ex_eta=log} holds by choosing $\ve>0$ small enough. 
On the other hand, if $\ga<1$, then noting that 
\begin{equation}
  \eta''(t) 
  \le \frac{\ga \(\log\log\(e^e+t\)\)^{\ga-1}}{\(e^e+t\)\log\(e^e+t\)}
    \(1+\frac{\ga-1}{e}\)
  \le 
    \ga \frac{\(\log\log\(e^e+t\)\)^{\ga-1}}{(1+t)\log\(e^e+t\)}, 
\end{equation}
we cannot expect the energy decay 
$\lim_{t\to\infty} E(t) \to 0$ from Theorem \ref{Thm1} 
since $\lim_{t\to\infty} \mu(t)/\log t = \infty$. 
\end{example}

In Theorem \ref{Thm1}, we considered the case that the condition to the initial data is weakened; which actually relates to the influence of low frequency energy on the energy decay estimate. 
On the other hand, the effect of high-frequency energy is thought to be almost negligible when $b(t)$ is monotonic, but it may have a crucial effect when $b(t)$ has oscillations. 
The effect of such oscillations of $b(t)$ on high-frequency energy is possible to be controlled by considering the initial data in $H^\infty$ such as the Gevrey class (see \cite{Ro93}); for precise, see \cite{GH25} and concluding remarks. 
The next theorem asserts that the loss of decay that may occur in Theorem \ref{Thm1} as $\lim_{t\to\infty}\mu(t)=\infty$, can be avoided by setting stronger restrictions to the initial data. 

\begin{theorem}\label{Thm2}
Let (A1), (A2) and (A3) be valid. 
For any $\ve>0$, there exist positive constants $C$ and $\nu$ such that the following estimate holds: 
\begin{equation}\label{Eest-Thm2}
  E(t) \le C (1+t)^{-\tm_0} 
  \int_{\R^d} 
  \exp\(2\nu \zeta(2|\xi|+\ve)\)
    \(\(1+|\xi|^2\)|\hat{u}_0(\xi)|^2 + |\hat{u}_1(\xi)|^2\)\,d\xi, 
\end{equation}
where 
\[
  \zeta(r) = \sqrt{\mu \((\eta')^{-1}(r)\)} 
\]
and $\hat{u}_j(\xi)$ $(j=0,1)$ denote the Fourier transform of $u_j(x)$ with respect to spacial variable $x \in \R^d$.
\end{theorem}

\begin{example}
Let $\eta(t)$ be defined in Example \ref{ex_eta=log} for $\be<1$. 
Since 
\begin{align*}
  \la 
 =\eta'(t)
 =\(\log(e+t)\)^{\be}
  \(1 + \frac{\be}{\log(e+t)}\)
 \ge \(\log(e+t)\)^{\be}, 
\end{align*}
it follows that 
\[
  \log\(e + (\eta')^{-1}(\la)\) \le \la^{\frac{1}{\be}}, 
\]
we see that 
\begin{align*}
  \mu\((\eta')^{-1}(\la)\)
  \le 
  \frac{e}{\be^{2}} \(\log\(e+(\eta')^{-1}(\la)\)\)^{1-\be}
  \le \frac{e}{\be^{2}} \la^{\frac{1-\be}{\be}}. 
\end{align*}
Therefore, we have 
\begin{align*}
  \exp\(\nu \zeta(2|\xi|+\ve)\)
  \le
  \exp\(\frac{\sqrt{e} \nu}{\be} (2|\xi|+\ve)^{\frac{1-\be}{2\be}}\). 
\end{align*}
Consequently, 
if $u_0, u_1 \in G^{s}_\infty$ with $s=2\be/(1-\be)$ and $\be>1/3$, 
where $G^s_\infty$ is the Gevrey class of order $s$ define by 
\[
  G^s_\infty:=\bigcap_{\ka>0}
  \left\{f\in L^2\;;\;
  \exp(\ka |\xi|^{1/s})\hat{f}(\xi) \in L^2
  \right\},
\]
then there exists a positive constant $C=C(u_0,u_1)$ such that the decay estimate without loss of decay $E(u) \le C(1+t)^{-\tm}$ holds by Theorem \ref{Thm2}. 
\end{example}

\begin{remark} 
If we apply the example in Example \ref{ex_eta=loglog} for $\ga<1$ to Theorem \ref{Thm2}, then 
\[
  \la = \eta'(t) \ge (\log\log(e^e + t))^\ga 
\]
and  
\[
  \mu((\eta')^{-1}(\la)) \le 
  \frac{e^{e+1}}{\ga(e-2)}\la^{\frac{1-\ga}{\ga}} \exp\(\la^{\frac{1}{\ga}}\), 
\]
thus $\hat{u}_0(\xi)$ and $\hat{u}_1(\xi)$ must be decreasing very fast as $|\xi|\to\infty$, which can no longer be described by the Gevrey class. 
\end{remark}

In the end of this section, we summarize the advances made by our main theorems over previous results. 
Actually, our perturbation term $\de(t)$ in \eqref{de(t)} seems to be very special form compared to those considered in Theorem \ref{ThmGH}; but doesn't require differentiability. 
Let us consider the following example $\de(t)$: 
\begin{equation}
  \de(t) = \frac{\sin\((1+t)^\al\)}{(1+t)^{p}}, 
\end{equation}
which is introduced in \cite{GH25} as an example for Theorem \ref{ThmGH}, 
where $\al>1$ and $p \le 1$. 
Indeed, \eqref{ThmGH-e1} and \eqref{ThmGH-e2} hold for 
\begin{equation}\label{ThmGH-albe}
  \be = -\al+1+\frac{p+\al-1}{k+1} \;\;(k \ge 1)
  \;\;\text{ \textit{and} }\;\;
  \ga=\al-1
\end{equation}
respectively. 
Then \eqref{ThmGH-bega} is represented by $(p-1)/(k+1) \ge 0$, and it also holds for $p=1$. 
Therefore, the assumption $k \ge 1$ of Theorem \ref{ThmGH} can be weakened by Theorem \ref{Thm1} in special cases. 
Next, let us consider the relationship with Theorem \ref{ThmGG}. 
The assumption of smoothness of $\de(t)$ in (i) of Theorem \ref{ThmGG} is weaker than the assumption required in Theorem \ref{Thm1}. 
However, if we apply our model \eqref{b=si} to Theorem \ref{ThmGG} (i) 
with $\de(t)=\si(\eta(t))+\ka_-$, 
where $\ka_-:=|\min_{\eta \in \R}\{\si(\eta)\}|$, 
that is, 
\begin{equation}
  b(t) 
 =\frac{m}{1+t} + \frac{\si(\eta(t))}{1+t}
 =\frac{m-\ka_-}{1+t} 
  + \frac{\de(t)}{1+t}, 
\end{equation}
the decay estimate \eqref{Eest0} holds with 
$\tm = \min\{2, m-\ka_-\}$, 
which is worse estimate than that obtained by Theorem \ref{Thm1} 
when $m<2+\ka_-$. 
It is clear that Theorem \ref{Thm1} is a generalization of Theorem \ref{ThmGG} (ii), and considers more details for $\al \to 1+0$ as in Example \ref{ex_eta=log} and Example \ref{ex_eta=loglog}. 
Moreover, Theorem \ref{Thm1} does not require $b(t) \ge 0$, as required in Theorem \ref{ThmGG}, under the assumption \eqref{0<m<1}. 
Finally, as mentioned in Remark \ref{rem-IC-H}, Theorem \ref{Thm1} shows the possibility that the initial condition may be weakened from 
$(u_0,u_1) \in H^1\cap L^2$ which is usually assumed in the previous works. 

%
%
\section{Proof of Theorem \ref{Thm1}}
%
%

%
\subsection{Representation of the solution in time-frequency space}
%

By partial Fourier transformation with respect $x \in \R^d$, and denoting
\begin{equation}
  v(t,\xi):=\hat{u}(t,\xi), 
  \;\;
  v_0(\xi):=\hat{u}_0(\xi)
  \;\;\text{ and }\;\;
  v_1(\xi):=\hat{u}_1(\xi),
\end{equation}
\eqref{u} is rewritten as follows: 
\begin{equation}\label{v}
\begin{cases}
  \(\pa_t^2 + |\xi|^2 + b(t) \pa_t\) v(t,\xi) = 0,
  & (t,\xi) \in (0,\infty) \times \R^d,\\ 
  v(0,\xi)=v_0(\xi),\;\; \pa_t v(0,\xi)) = v_1(\xi), & \xi \in \R^d. 
\end{cases}
\end{equation}
Denoting the energy density to the solution of \eqref{v} at $t$ by 
\begin{equation}
  \cE(t,\xi)=\cE(t,\xi;v):=
  |\xi|^2 |v(t,\xi)|^2 + \left|\pa_t v(t,\xi)\right|^2, 
\end{equation}
we see that 
\begin{equation}\label{E-cE}
  E(t;u)=\int_{\R^d} \cE(t,\xi;v)\,d\xi 
\end{equation}
by Parseval's identity. 
So we shall derive the total energy estimate \eqref{Eest-Thm} by proving an uniform estimate of $\cE(t,\xi)$ with respect to $\xi$. 
More precisely, $\cE(t,\xi)$ will be estimated with appropriate methods in each region of time-frequency space $\{(t,\xi)\;;\; t\in[0,\infty),\ \xi \in \R^d\}$ divided by the hyper surface $t=\cT_0=\cT_0(\xi;N)$ defined by 
\begin{equation}\label{def_t0}
  \cT_0(\xi;N):= \max\left\{N|\xi|^{-1}-1,0\right\} 
\end{equation}
with a real number $N \ge 1$. 
Here we define \textit{hyperbolic zone} $Z_H=Z_H(N)$ and 
\textit{dissipative zone} (or \textit{parabolic zone}) $Z_D=Z_D(N)$ by 
\begin{equation}\label{Z_H}
  Z_{H}(N):=\left\{(t,\xi) \in [0,\infty) \times \R^d \;;\; 
  t \ge \cT_0(\xi;N)\right\} 
\end{equation}
and 
\begin{equation}\label{Z_D}
  Z_D(N)
 :=\left\{
  (t,\xi) \in [0,\infty) \times \R^d\;;\;
  t \le \cT_0(\xi;N) \right\}.
\end{equation}
Moreover, $Z_H$ will be divided by 
\begin{equation}\label{Z_H1}
  Z_{H1}(N):=
  \left\{(t,\xi) \in [0,\infty) \times B_N 
    \;;\; t \ge \cT_0(\xi;N) \right\} 
\end{equation}
and 
\begin{equation}\label{Z_H2}
  Z_{H2}(N):=\left\{(t,\xi) \in [0,\infty) \times \R^d \setminus B_N \;;\; 
  t \ge \cT_0(\xi;N)\right\}, 
\end{equation}
where 
\[
  B_N:=\left\{ \xi \in \R^d \;;\; |\xi| < N\right\}. 
\]
Here we note that 
$(1+t)|\xi| \le N$, 
$(1+t)|\xi| \ge N$, and $|\xi|\ge N$ hold in 
$Z_D$, $Z_{H1}$, and $Z_{H2}$, respectively.  


\begin{proposition}\label{Prop1}
Assume that (A1), (A2) and (A3) and valid. 
\begin{itemize}
\item[(i)] 
Let $(t,\xi) \in Z_{H}$. 
For any $\ve>0$, there exists a positive constant $K_H$ such that the following estimates hold: 
\begin{equation}\label{est_cE-ZH}
  \cE(t,\xi) 
  \le 
  \begin{cases}
    K_H \(\dfrac{1+t}{1+\cT_0}\)^{-m} \cE(\cT_0,\xi)
      & \text{\textit{ in }}\;\; Z_{H1},
\\
    K_H \exp\(\ve\mu(t)\)(1+t)^{-m} \cE(0,\xi)
      & \text{\textit{ in }}\;\; Z_{H2}. 
  \end{cases}
\end{equation}
\item[(ii)] 
Let $(t,\xi) \in Z_D$. 
If \eqref{b>=0} is valid, then there exists a positive constant $K_D$ such that the following estimate holds: 
\begin{equation}\label{est_cE-ZD1}
  \cE(t,\xi) 
  \le 
  K_D \(|\xi|^2 |v_0(\xi)|^2 
  + \((1+t)^{-2m}+|\xi|^2 \ga(t;m)^2\)|v_1(\xi)|^2\), 
\end{equation}
where 
\begin{equation}
  \ga(t;m)=\int^t_0 (1+s)^{-m}\,ds. 
\end{equation}
\item[(iii)] 
Let $(t,\xi) \in Z_D$. 
If \eqref{0<m<1} is valid, then there exists a positive constant $K_D$ such that the following estimate holds: 
\begin{equation}\label{est_cE-ZD2}
  \cE(t,\xi) 
  \le 
  K_D \(|\xi|^2 |v_0(\xi)|^2 + (1+t)^{-2m}|v_1(\xi)|^2\). 
\end{equation}
\end{itemize}
\end{proposition}

Denoting the solutions to \eqref{v} with the initial data 
$(\Re v_0, \Re v_1)$ and $(\Im v_0, \Im v_1)$ 
by $v_\Re$ and $v_\Im$ respectively, 
the solution $v$ to \eqref{v} with complex-valued initial data 
$(v_0,v_1)$ is represented by 
$v=v_\Re + i v_\Im$. 
Then we have 
\begin{align*}
  \cE(t,\xi;v)
 =\cE(t,\xi;v_\Re)+\cE(t,\xi;v_\Im), 
\end{align*}
therefore, we can restrict ourselves to the case that the initial data and the solution to \eqref{v} are real-valued without loss of generality. 

We introduce the polar coordinate representation to the solution of
\begin{equation}\label{ve}
  \(\pa_t^2 + |\xi|^2 + b(t) \pa_t\) v(t,\xi) = 0,
\end{equation}
which was introduced in \cite{GG25}. 
Let $v=v(t,\xi)$ be a solution to the equation of \eqref{ve}, and consider the following polar coordinate representation: 
\begin{equation}\label{v-rho-th}
  \begin{pmatrix} 
    |\xi|v(t,\xi) \\ \pa_t v(t,\xi)
  \end{pmatrix}
 =\begin{pmatrix}
    \rho \cos \th \\ \rho \sin \th
  \end{pmatrix},
\end{equation}
where $\th=\th(t,\xi)$ is real valued, 
and $\rho=\rho(t,\xi)$ is non-negative. 
Then we have the following lemma. 

\begin{lemma}[Ghisi-Gobbino\cite{GG25}]\label{lemma_rho-th}
Let $0\le t_0 < t$ and $\cE(t_0,\xi)>0$. 
$\th$ and $\rho$ defied by \eqref{v-rho-th} are the solutions to the following initial value problem: 
\begin{equation}\label{ode_rho-th}
\begin{dcases}
  \pa_t \rho(t,\xi) = -b(t) \rho(t,\xi) \sin^2(\th(t,\xi)), \\
  \pa_t \th(t,\xi)
   =-|\xi| - b(t) \sin(\th(t,\xi))\cos(\th(t,\xi)),
\end{dcases}
\end{equation}
\begin{equation}\label{ic_rho-th}
  \rho(t_0,\xi) = \sqrt{\cE(t_0,\xi)},
  \quad
  \th(t_0,\xi) = \cos^{-1}\(\frac{|\xi|v(t_0,\xi)}{\sqrt{\cE(t_0,\xi)}}\). 
\end{equation}
Moreover, we have 
\begin{equation}\label{cE-th}
  \rho(t,\xi)^2
 =\cE(t,\xi)=\cE(t_0,\xi) \exp\(-2\int^t_{t_0} b(s)\sin^2(\th(s,\xi))\,ds\). 
\end{equation}
\end{lemma}
\begin{proof} 
$\rho(t,\xi)^2=\cE(t,\xi)$ is clear from the definition. 
Differentiating both sides of $\rho(t,\xi)^2 = \cE(t,\xi)$ with respect to $t$, we have 
\begin{align*}
2\rho \pa_t \rho
 =2 \(\pa_t^2 v\)\pa_t v + 2|\xi|^2 \(\pa_t v\)v
 =-2b(t)(\pa_t v)^2
 =-2b(t)\rho^2 \sin^2\th, 
\end{align*}
it follows that 
\begin{equation}\label{rho_t-rho}
  \pa_t \rho(t,\xi) = -b(t)\rho(t,\xi) \sin^2\(\th(t,\xi)\),
\end{equation}
that is, 
\begin{equation}\label{pa_t-cE}
  \pa_t \cE(t,\xi) = -2b(t) \sin^2(\th(t,\xi)) \cE(t,\xi). 
\end{equation}
Here \eqref{cE-th} is derived as the solution of \eqref{pa_t-cE}. 
Moreover, noting that 
\begin{align*}
  (\pa_t \rho)\sin \th
 =-b(t)\rho\sin^3\th
 =b(t)\rho \cos^2\th \sin\th -b(t)\rho \sin \th, 
\end{align*}
we have 
\begin{align*}
  0
&=\pa_t^2 v + |\xi|^2 v + b(t)\pa_t v
\\
&=(\pa_t \rho)\sin \th + \rho(\pa_t \th)\cos\th
 +|\xi|\rho \cos\th + b(t)\rho\sin\th
\\
&=b(t)\rho \cos^2\th \sin\th 
 + \rho(\pa_t \th)\cos\th
 +|\xi|\rho \cos\th 
\\
&=\rho\cos\th
  \( b(t) \cos\th \sin\th + \pa_t \th + |\xi| \), 
\end{align*}
it follows that 
\begin{equation}\label{th_t-|xi|}
  \pa_t \th = -|\xi| - b(t) \cos\th \sin\th. 
\end{equation}
\end{proof}

A remarkable point of Lemma \ref{lemma_rho-th} is that the effect of dissipation is explicitly described by an integral. 
Actually, there are problems that $\th(t,\xi)$ is a solution to a non-linear problem, and that $\cE(t,\xi;v)$ must be expressed by $\cE(t_0,\xi;v)$, in other words, it may not be useful to consider near $|\xi|=0$, 
but in any cases, the representation \eqref{cE-th} is extremely useful for precise estimates in a specific region of the time-frequency space. 

%
\subsection{Proof of Proposition \ref{Prop1} (i)}
%
The following proof of Proposition \ref{Prop1}  (i)  is an improvement of the method used in the proof of Theorem \ref{ThmGG}.

From now on, we may assume that $\si(\eta)$ is $2\pi$-periodic function without loss of generality, and denote 
\[
  {\cs}_1(\eta):=\cos(\eta)
  \;\;\text{ \textit{and} }\;\;
  {\cs}_2(\eta):=\sin(\eta)
\]
for convenience. 

\begin{lemma}\label{Lemma_ossint}
Let $t_0 \ge 0$, $n \in \N$ and $[\tau_{-}, \tau_{+}] \subset [t_0,\infty)$. 
If $\phi \in C^2([\tau_{-},\tau_{+}])$ and 
$\psi \in C^1([\tau_{-},\tau_{+}])$ satisfy
\begin{equation}\label{ass_phi-psu}
  |\phi'(t)| \ge \phi_0, \;\;
  \phi''(t) \ge 0
  \;\;\text{ and }\;\;
  |\psi'(t)| \le \frac{\Psi_0}{1+t}
\end{equation}
for positive real numbers $\phi_0$ and $\Psi_0$, 
then the following inequalities hold: 
\begin{equation}
  \left|\int_{\tau_{-}}^{\tau_{+}} 
  \frac{\cs_j(n \phi(s)) \cs_k(\psi(s))}{1+s}\,ds\right|
  \le 
  \frac{\Psi_0+4}{n\phi_0(1+t_0)}
  \quad (j,k = 1,2). 
\end{equation}
\end{lemma}
%
\begin{proof}
Integrating by parts, we have 
\begin{align*}
  n\int_{\tau_{-}}^{\tau_{+}} 
  \frac{\cos(n\phi(s)) \sin(\psi(s))}{1+s}\,ds
&=\(\sin(n\phi(\tau_{+}))\frac{\sin(\psi(\tau_{+}))}{(1+\tau_{+})\phi'(\tau_{+})}
 -\sin(n\phi(\tau_{-}))\frac{\sin(\psi(\tau_{-}))}{(1+\tau_{-})\phi'(\tau_{-})}\)
\\
&\quad
 -\int_{\tau_{-}}^{\tau_{+}} 
  \sin(n\phi(s)) \frac{\psi'(s)\cos(\psi(s))}{(1+s)\phi'(s)}\,ds
\\
&\quad +\int_{\tau_{-}}^{\tau_{+}} 
  \sin(n\phi(s)) \frac{\sin(\psi(s))}{(1+s)^2 \phi'(s)}\,ds
\\
&\quad +\int_{\tau_{-}}^{\tau_{+}} 
  \sin(n\phi(s)) \frac{\sin(\psi(s)) \phi''(s)}{(1+s) (\phi'(s))^2}\,ds
\\
&=:I_1+I_2+I_3+I_4. 
\end{align*}
By \eqref{ass_phi-psu}, we have 
\begin{align*}
  |I_1| \le \frac{1}{\phi_0(1+\tau_{+})}+ \frac{1}{\phi_0(1+\tau_{-})}
  \le \frac{2}{\phi_0(1+t_0)},
\end{align*}
\begin{align*}
  |I_2| 
  \le \int_{\tau_{-}}^{\tau_{+}} \frac{\Psi_0}{\phi_0 (1+s)^2}\,ds
  = \frac{\Psi_0}{\phi_0}\(\frac{1}{1+\tau_{-}}-\frac{1}{1+\tau_{+}}\)
  \le \frac{\Psi_0}{\phi_0 (1+t_0)},
\end{align*}
\begin{align*}
  |I_3|
  \le \int_{\tau_{-}}^{\tau_{+}} \frac{1}{\phi_0 (1+s)^2}\,ds
  \le \frac{1}{\phi_0 (1+\tau_{-})}
  \le \frac{1}{\phi_0 (1+t_0)} 
\end{align*}
and 
\begin{align*}
  |I_4|
  \le 
  \frac{1}{1+\tau_{-}}\left|\int_{\tau_{-}}^{\tau_{+}} 
    \frac{\phi''(s)}{\phi'(s)^2}\,ds\right|
 =\frac{1}{1+t_0}
  \left|-\frac{1}{\phi'(\tau_{+})}+\frac{1}{\phi'(\tau_{-})}\right|
 \le \frac{1}{\phi_0 (1+t_0)}. 
\end{align*}
Summarizing the estimates above, we have 
\begin{align*}
 n\left|\int_{\tau_{-}}^{\tau_{+}} 
  \frac{\cos(n\phi(s)) \sin(\psi(s))}{1+s}\,ds\right|
  \le
  \frac{\Psi_0+4}{\phi_0(1+t_0)}.
\end{align*}
Analogously, we have 
\begin{align*}
 n\left|\int_{\tau_{-}}^{\tau_{+}} 
  \frac{\cs_j(n\phi(s)) \cs_k(\psi(s))}{1+s}\,ds\right|
  \le
  \frac{\Psi_0+4}{\phi_0(1+t_0)}
\end{align*}
for $j,k = 1,2$. 
\end{proof}

\begin{lemma}\label{Lemm-phi_n}
Let $n \in \N$, $t_0 \ge 0$, 
$\ka$ and $\la$ be positive real numbers. 
We define $\phi_n(t)$ on $[t_0,\infty)$ and 
$\tau_{n \pm}$ by 
\[
  \phi_n(t):= \eta(t)-\frac{2 \la t}{n}, 
\]
\begin{equation}\label{def_tau-}
  \tau_{n-}
 :=\begin{dcases}
    t_0 
      &\text{ \textit{ if} }\; \eta'(t_0) > \frac{2\la - \ka}{n},\\
    (\eta')^{-1}\(\frac{2\la-\ka}{n}\)
      &\text{ \textit{ if} }\; \eta'(t_0) \le \frac{2\la - \ka}{n}
  \end{dcases}
\end{equation}
and 
\begin{equation}\label{def_tau+}
  \tau_{n+}
 :=\begin{dcases}
    t_0 
      &\text{ \textit{ if} }\; \eta'(t_0) > \frac{2\la + \ka}{n},\\
    (\eta')^{-1}\(\frac{2\la+\ka}{n}\)
      &\text{ \textit{ if} }\; \eta'(t_0) \le \frac{2\la + \ka}{n}
  \end{dcases}
\end{equation}
respectively. 
Then the following inequalities hold: 
\begin{equation}\label{eq1-Lemm-phi_n}
  \phi_n'(t) 
  \begin{dcases}
    \le -\frac{\ka}{n} & (t \le \tau_{n-}), \\ \ge \frac{\ka}{n} & (t \ge \tau_{n+}).
  \end{dcases}
\end{equation}
\end{lemma}
%
\begin{proof}
We note that $(\eta')^{-1}$ is strictly increasing since $\eta'$ is strictly increasing by (A3). 
If $t \le \tau_{n-}$, and thus $t_0 < \tau_{n-}$, then $\eta'(\tau_{n-})=(2\la-\ka)/n$, and thus
\begin{align*}
  \phi_n'(t)
=\eta'(t) - \frac{2\la}{n} 
=-\frac{\ka}{n} - \(\eta'(\tau_{n-}) - \eta'(t)\)
\le -\frac{\ka}{n}. 
\end{align*}
Analogously, if $t_0 \le \tau_{n+} < t$, then 
$\eta'(\tau_{n+})=(2\la+\ka)/n$, and thus 
\begin{align*}
  \phi_n'(t)
=\eta'(t) - \frac{2\la}{n} 
=\frac{\ka}{n} - \(\eta'(\tau_{n+}) - \eta'(t)\)
\ge \frac{\ka}{n}. 
\end{align*}
If $\tau_{n+} \le t_0 < t$, then $\eta'(\tau_{n+})=\eta'(t_0) \ge (2\la+\ka)/n$, and thus
\begin{align*}
  \phi_n'(t)
 =\eta'(t) - \frac{2\la}{n} 
 > \eta'(t_0) - \frac{2\la}{n} 
 > \frac{2\la + \ka}{n} - \frac{2\la}{n} 
 = \frac{\ka}{n}. 
\end{align*}
\end{proof}

\begin{lemma}\label{Lemm-log(tau+/tau-)}
Let $n\in \N$ and $\tau_{n\pm}$ be defined by \eqref{def_tau-} and \eqref{def_tau+}. 
For any $\tau_-$ and $\tau_+$ satisfy 
$\tau_{n-} \le \tau_- < \tau_+ \le \tau_{n+}$, 
the following inequality holds: 
\[
  \log \(\frac{1+\tau_{+}}{1+\tau_{-}}\) 
  \le \frac{2\ka}{n}\mu(\tau_+). 
\]
\end{lemma}
\begin{proof}
$\tau_+$ can be represented by $\nu \in (-1,1]$ as follow: 
\[
  \tau_+ = (\eta')^{-1}\(\frac{2\la + \nu \ka}{n}\). 
\]
Here we note that the following equality holds: 
\begin{align*}
  \left. \frac{d}{ds}\(\eta'\)^{-1}(s)\right|_{s=\eta'(t)}
 =\frac{1}{\eta''(t)}. 
\end{align*}
If $t_0 <\tau_{n-}$, then by mean value theorem, 
there exists $\tilde{s}_1 \in ((2\la-\ka)/n,(2\la+\nu\ka)/n)$, 
that is, $\tilde{\tau}_1:=(\eta')^{-1}(\tilde{s}_1) \in (\tau_{n-},\tau_+)$,  such that
\begin{align*}
  \log\(\frac{1+\tau_{+}}{1+\tau_{-}}\)
&\le \log\(\frac{1+\tau_{+}}{1+\tau_{n-}}\)
\\
&=\log \(1+(\eta')^{-1}\(\dfrac{2\la + \nu\ka}{n}\)\) 
    - \log \(1+(\eta')^{-1}\(\dfrac{2\la - \ka}{n}\)\)
\\
&=\(\frac{2\la + \nu\ka}{n} - \frac{2\la - \ka}{n}\)
  \frac{\log \(1+(\eta')^{-1}\(\dfrac{2\la + \nu\ka}{n}\)\) 
    - \log \(1+(\eta')^{-1}\(\dfrac{2\la - \ka}{n}\)\)}
  {\dfrac{2\la + \nu\ka}{n} - \dfrac{2\la - \ka}{n}}
\\
&=\frac{(1+\nu)\ka}{n}
  \left.\dfrac{d}{ds}\log \(1+(\eta')^{-1}(s)\)
  \right|_{s=\tilde{s}_1}
 =\frac{(1+\nu)\ka}{n}
  \left.\frac{\dfrac{d}{ds}(\eta')^{-1}(s)}{1+(\eta')^{-1}(s)}
  \right|_{s=\tilde{s}_1}
\\
&=\frac{(1+\nu)\ka}{n}
  \frac{1}{\(1+\tilde{\tau}_1\) 
    \eta''\(\tilde{\tau}_1\)}
\le \frac{2\ka}{n}\mu\(\tilde{\tau}_1\) 
\\
&\le \frac{2\ka}{n}\mu\(\tau_{+}\). 
\end{align*}
If $\tau_{n-} = t_0 < \tau_{n+}$, and thus $\eta'(t_0)>(2\la-\ka)/n$, 
then there exists 
$\tilde{s}_2 \in (\eta'(t_0),(2\la+\nu\ka)/n)$, that is, 
$\tilde{\tau}_2:=(\eta')^{-1}(\tilde{s}_2) \in (t_0,\tau_+)$, 
such that 
\begin{align*}
  \log\(\frac{1+\tau_{+}}{1+\tau_{-}}\)
&\le  \log\(\frac{1+\tau_{+}}{1+t_0}\)
\\
&=\log \(1+(\eta')^{-1}\(\dfrac{2\la + \nu\ka}{n}\)\) 
    - \log \(1+(\eta')^{-1}\(\eta'(t_0)\)\)
\\
&=\(\frac{2\la + \nu\ka}{n} - \eta'(t_0)\)
  \frac{\log \(1+(\eta')^{-1}\(\dfrac{2\la + \nu\ka}{n}\)\) 
    - \log \(1+(\eta')^{-1}\(\eta'(t_0)\)\)}
  {\dfrac{2\la + \nu\ka}{n} - \eta'(t_0)}
\\
&<\(\frac{2\la + \nu\ka}{n} - \frac{2\la - \ka}{n}\)
  \left.\dfrac{d}{ds}\log \(1+(\eta')^{-1}(s)\)
  \right|_{s=\tilde{s}_2}
 =\frac{(1+\nu)\ka}{n}
  \left.\frac{\dfrac{d}{ds}(\eta')^{-1}(s)}{1+(\eta')^{-1}(s)}
  \right|_{s=\tilde{s}_2}
\\
&=\frac{(1+\nu)\ka}{n}
  \frac{1}{\(1+\tilde{\tau}_2\) 
           \eta''\(\tilde{\tau}_2\)}
\le \frac{2\ka}{n}\mu\(\tilde{\tau}_2\)
\\
& \le \frac{2\ka}{n}\mu(\tau_+). 
\end{align*}
If $t_0>\tau_{n+}$, that is, 
$\tau_{n+}=\tau_{n-}=t_0$, then 
$\log((1+\tau_{n+})/(1+\tau_{n-})) = 0$. 
\end{proof}

\begin{lemma}\label{lemma_int-b-th}
Let $t_0 \ge 0$ and $h(t)\in C^1([t_0,\infty))$ satisfy 
\begin{equation}
  H_0:=\sup_{t \ge t_0}\{(1+t)|h'(t)|\}<\infty.
\end{equation}
Then, for any positive real numbers $\ka$ and $\la$, 
the following estimate holds: 
\begin{equation}\label{eq1-lemma_int-b-th}
  \left|\int^t_{t_0}
  \frac{\si_0(\eta(s)) \cos(2\la s + 2 h(s))}{1+s}\,ds
  \right|
  \le 
  \frac{(H_0+2)(2m+A_1 )}{\la(1+t_0)}
  +\frac{4A_1(H_0+2)}{\ka}
  +2\ka A_1  \mu(t). 
\end{equation}
In particular, if $\la \le N$ for a positive constant $N$, then 
there exists a positive constant $\tau_{N,\ka}$ such that 
\begin{equation}
  \left|\int^t_{t_0}
  \frac{\si_0(\eta(s)) \cos(2\la s + 2 h(s))}{1+s}\,ds
  \right|
  \le 
  \frac{(H_0+2)(2m+A_1 )}{\la(1+t_0)}
  +\frac{4A_1(H_0+2)}{\ka}
  +2\ka A_1  \mu(\tau_{N,\ka}). 
\end{equation}
\end{lemma}
\begin{proof}
Let $n \in \N$. 
We define $g_{j,n}(s)$ for $j=1,\ldots,4$ by 
\begin{align*}
& g_{1,n}(s):=\sin(n\eta(s)+2\la s) \cos(2 h(s)),
\\
&  g_{2,n}(s):=\sin(n\eta(s)-2\la s) \cos(2 h(s)),
\\
& g_{3,n}(s):=\cos(n\eta(s)+2\la s) \sin(2 h(s)),
\\
&  g_{4,n}(s):=-\cos(n\eta(s)-2\la s) \sin(2 h(s)),
\end{align*}
so that 
\begin{align*}
  2\sin(n\eta(s))\cos(2\la s + 2 h(s))
 =g_{1,n}(s)+g_{2,n}(s)+g_{3,n}(s)+g_{4,n}(s). 
\end{align*}
By applying Lemma \ref{Lemma_ossint} with 
$\tau_-=t_0$, $\tau_+=t$,
\[
  \phi(t)=\eta(t)+\frac{2\la t}{n},\;\;
  \psi(t)=2h(t),\;\;
  \phi_0 = \frac{2 \la}{n},\;\;
  \Psi_0 = 2 H_0, 
\]
we have 
\begin{align*}
  \left|\int^t_{t_0}\frac{g_{j,n}(s)}{1+s}\,ds\right|
 \le \frac{H_0+2}{\la(1+t_0)} 
\end{align*}
for $j=1,3$. 
Let $\tau_-$ and $\tau_+$ satisfy $t_0 \le \tau_- < \tau_+$. 
If either 
$\tau_+ < \tau_{n-}$ or $\tau_{n+} < \tau_-$ holds, 
then by applying Lemma \ref{Lemma_ossint} and Lemma \ref{Lemm-phi_n} with 
\[
  \phi(t)=\phi_n(t),\;\;
  \psi(t)=2h(t),\;\;
  \phi_0 = \frac{\ka}{n},\;\;
  \Psi_0 = 2H_0,  
\]
we have 
\begin{align*}
  \left|\int^{\tau_+}_{\tau_-}\frac{g_{2,n}(s)}{1+s}\,ds\right|
 =\left|\int^{\tau_+}_{\tau_-}
  \frac{\sin(n\phi_n(s)) \cos(2 h(s))}{1+s}\,ds \right|
 \le
  \frac{2H_0 + 4}{\ka(1+t_0)}
 \le
  \frac{2(H_0 + 2)}{\ka}. 
\end{align*}
If $[\tau_-,\tau_+] \subseteq [\tau_{n-},\tau_{n+}]$, 
then by Lemma \ref{Lemm-log(tau+/tau-)}, we have 
\begin{align*}
  \left|\int^{\tau_+}_{\tau_-}\frac{g_{2,n}(s)}{1+s}\,ds\right|
\le \int^{\tau_+}_{\tau_-} \frac{ds}{1+s}
= \log\(\frac{1+\tau_{+}}{1+\tau_{-}}\) 
  \le \frac{2\ka}{n}\mu(\tau_+)
  \le 2\ka \mu(\tau_+). 
\end{align*}
Since the same estimates hold for $g_{4,n}(s)$, we have 
\begin{align*}
  \left|\int^t_{t_0}\frac{g_{j,n}(s)}{1+s}\,ds\right|
  \le
  \begin{dcases}
    \frac{2(H_0 + 2)}{\ka} 
      & (t \le \tau_{n-}), \\
    2\(\ka\mu(t)
     + \frac{H_0 + 2}{\ka}\) 
      & (\tau_{n-} < t \le \tau_{n+}), \\
    2\( \ka\mu(t) + \frac{2(H_0 + 2)}{\ka}\) 
      & (\tau_{n+} < t)
  \end{dcases}
\end{align*}
for $j=2, 4$. 
Summarizing the estimates above, we have 
\begin{equation}
  \left|\int^t_{t_0}
  \frac{\sin(n\eta(s)) \cos(2\la s + 2 h(s))}{1+s}\,ds \right|
  \le 
  \frac{H_0+2}{\la(1+t_0)}
   +2\(\ka \mu(t) + \frac{2(H_0+2)}{\ka}\). 
\end{equation}
Analogously, we have 
\begin{equation}
  \left|\int^t_{t_0}
  \frac{\cos(n\eta(s)) \cos(2\la s + 2 h(s))}{1+s}\,ds \right|
  \le 
  \frac{H_0+2}{\la(1+t_0)}
   +2\(\ka \mu(t) + \frac{2(H_0+2)}{\ka}\)
\end{equation}
for any $n\in\N$. 
If $n=0$, then by applying Lemma \ref{Lemma_ossint} with 
$n=1$, 
$\tau_-=t_0$, $\tau_+=t$, 
\[
  \phi(t)=2\la t,\;\;
  \psi(t)=2h(t),\;\;
  \phi_0 = 2 \la
  \;\;\text{ \textit{and} }\;\;
  \Psi_0 = 2 H_0, 
\]
we have 
\begin{align*}
  \left|\int^t_{t_0}\frac{\cos(2\la s + 2 h(s))}{1+s}\,ds\right|
  &\le
    \left|\int^t_{t_0}\frac{\cos(2\la s)\cos(2 h(s))}{1+s}\,ds\right|
   +\left|\int^t_{t_0}\frac{\sin(2\la s)\sin(2 h(s))}{1+s}\,ds\right|,
  \\
  &\le \frac{2(H_0+2)}{\la(1+t_0)}. 
\end{align*}
Consequently, we have 
\begin{align*}
&  \left|\int^t_{t_0}
  \frac{\si_0(\eta(s)) \cos(2\la s + 2 h(s))}{1+s}\,ds
  \right|
\\
  &\le 
  \frac{2m(H_0+2)}{\la(1+t_0)}
 +\(
    \frac{H_0+2}{\la(1+t_0)}
   +2\(\ka \mu(t) + \frac{2(H_0+2)}{\ka}\)
  \)
  \sum_{n=1}^\infty \(|\al_n|+|\be_n|\)
\\
 &\le 
  \frac{(H_0+2)(2m+A_1 )}{\la(1+t_0)}
  +\frac{4A_1(H_0+2)}{\ka}
  +2\ka A_1  \mu(t). 
\end{align*}
In particular, if $\la \le N$, then setting $\tau_{N,\ka}$ such as
\begin{align*}
  \tau_{n+}
 &\le \max\left\{t_0, (\eta')^{-1}\(\frac{2N+\ka}{n}\)\right\}
  \le \max\left\{t_0, (\eta')^{-1}\(2N+\ka\)\right\}=:\tau_{N,\ka}, 
\end{align*}
we have 
\begin{align*}
  \left|\int^t_{t_0}\frac{g_{j,n}(s)}{1+s}\,ds\right|
  \le
    2\( \ka\mu(\tau_{N,\ka}) + \frac{2(H_0 + 2)}{\ka}\) 
\end{align*}
for $j=2, 4$. 
Therefore, $\mu(t)$ in the estimate \eqref{eq1-lemma_int-b-th} can be replaced with $\mu(\tau_{N,\ka})$. 
\end{proof}

\medskip
\noindent
\textit{Proof of Proposition \ref{Prop1} (i)}.\;
Let $(t,\xi) \in Z_{H}$. 
We define $h(t,\xi)$ by 
\[
  h(t,\xi) = - |\xi| t + \th(t,\xi). 
\]
Then we see that 
\begin{align*}
  |\pa_t h(t,\xi)| 
 = \left|-|\xi| + \pa_t \th(t,\xi)\right|
 = \left|b(t)\sin(\th(t,\xi))\cos(\th(t,\xi))\right|
 \le \frac{b_1}{1+t} 
\end{align*}
by \eqref{ode_rho-th}, 
where $b_1=m + \max_{\eta}\{|\si(\eta)|\}$. 
By applying Lemma \ref{lemma_int-b-th} with $h(t)=h(t,\xi)$, $H_0=b_1$, $\la=|\xi|$ and $t_0=\cT_0$, we have 
\begin{align*}
\left|\int^t_{\cT_0}
  \frac{\si(\eta(s)) \cos(2\th(s,\xi))}{1+s}\,ds
  \right|
&=\left|\int^t_{\cT_0}
  \frac{\si(\eta(s)) \cos(2|\xi| s + 2 h(s,\xi))}{1+s}\,ds
  \right|
\\
&\le 
  \frac{(b_1+2)(2m + A_1 )}{|\xi|(1+\cT_0)}
 +\frac{4A_1(b_1+2)}{\ka} 
 +2\ka A_1 \mu(t)
\\
&\le 
  \frac{(b_1+2)(2m + A_1 )}{N}
 +\frac{4A_1(b_1+2)}{\ka} 
 +2\ka A_1 \mu(t).
\end{align*}
Therefore, recalling \eqref{B0}, we have 
\begin{equation}\label{-2intbsin^2}
\begin{split}
&  -2\int^t_{\cT_0}b(s)\sin^2\(\th(s,\xi)\)\,ds
\\
&=\int^t_{\cT_0}\frac{\si_0(\eta(s))}{1+s}
  \(-1+\cos\(2\th(s,\xi)\)\)\,ds
\\
&=-\int^t_{\cT_0}\frac{m}{1+s}\,ds
  -\int^t_{\cT_0}\frac{\si(\eta(s))}{1+s}\,ds
  +\int^t_{\cT_0}\frac{\si_0(\eta(s))\cos\(2\th(s,\xi)\)}{1+s}\,ds
\\
&\le 
   \log\(\frac{1+t}{1+\cT_0}\)^{-m}
  +B_0
 +\frac{(b_1+2)(2m + A_1 )}{N}
 +\frac{4A_1(b_1+2)}{\ka} 
 +2\ka A_1 \mu(t).
\end{split}
\end{equation}
Here we note that $\cT_0=0$ in $Z_{H2}$. 
In particular, if $(t,\xi) \in Z_{H1}$, that is, 
$|\xi|\le N$, then we have 
\begin{equation}\label{-2intbsin^2-ZH1}
\begin{split}
& -2\int^t_{\cT_0}b(s)\sin^2\(\th(s,\xi)\)\,ds
\\
&\quad
  \le 
   \log\(\frac{1+t}{1+\cT_0}\)^{-m}
  +B_0
 +\frac{(b_1+2)(2m + A_1 )}{N}
 +\frac{4A_1(b_1+2)}{\ka} 
 +2\ka A_1 \mu(\tau_{N,\ka}).
\end{split}
\end{equation}
Consequently, setting 
\[
  K_H
 =
\exp\(B_0
 +\frac{(b_1+2)(2m + A_1 )}{N}
 +\frac{4A_1(b_1+2)}{\ka}
 +2\ka A_1 \mu(\tau_{N,\ka})
  \)
\]
and 
\[
  \ka = \frac{\ve}{2A_1 }, 
\]
we have \eqref{est_cE-ZH}.
\qed

%
\subsection{Proof of Proposition \ref{Prop1} (ii)}
%

The following proof of Proposition \ref{Prop1} (ii) below is based on \cite{GG25}. 
In the proof, we don't use the non-effective condition \eqref{0<m<1}, 
but essentially use the non-negative dissipative condition \eqref{b>=0}. 


%

If $|\xi|=0$, then from $\pa_t^2 v + b(t)\pa_t v = 0$ the following equalities hold: 
\begin{align*}
  \pa_t v(t,0)
 &=v_1(0) \exp\(-\int^t_0 b(s)\,ds\)
\\
 &=v_1(0) \exp\(-\int^t_0 \frac{\si(\eta(s))}{1+s}\,ds\) (1+t)^{-m}. 
\end{align*}
Therefore, we have 
\begin{align*}
  \cE(t,0)
&=|\pa_t v(t,0)|^2
 =|v_1(0)|^2 \exp\(-2\int^t_0 \frac{\si(\eta(s))}{1+s}\,ds\) (1+t)^{-2m} 
\\
& \le e^{2B_0}(1+t)^{-2m} |v_1(0)|^2. 
\end{align*}
Thus \eqref{est_cE-ZD1} holds for $K_D \ge e^{2B_0}$. 

For $|\xi|>0$, we introduce the following lemma. 
\begin{lemma}\label{lemma_Prop1(ii)}
Let $|\xi|>0$, 
$\rho_j(t,\xi)$ and $\th_j(t,\xi)$ for $j=1,2$ be solutions of 
\eqref{ode_rho-th} with initial data
\begin{equation}\label{ic_rho1-th1}
  \rho_1(0,\xi) = |\xi||v_0(\xi)|, \quad
  \th_1(0,\xi) = 
    \begin{cases} 
      0 & (v_0(\xi) \ge 0), \\ \pi & (v_0(\xi)<0), 
    \end{cases}
\end{equation}
\begin{equation}\label{ic_rho2-th2}
  \rho_2(0,\xi) = |v_1(\xi)|, \quad
  \th_2(0,\xi) = 
    \begin{cases} 
      \pi/2 & (v_1(\xi) \ge 0), \\ -\pi/2 & (v_1(\xi)<0). 
    \end{cases}
\end{equation}
If (A1), (A2) with \eqref{b>=0} and (A3) are valid, 
then the following estimates hold: 
\begin{equation}\label{est_rhoj}
  \rho_j(t,\xi)^2 \le 
  \begin{cases}
    |\xi|^2 |v_0(\xi)|^2 & \text{ \textit{for} }\;\; j=1,\\ 
    e^{2B_0}\((1+t)^{-2m}+|\xi|^2\ga(t;m)^2\)|v_1(\xi)|^2 
      & \text{ \textit{for} }\;\; j=2. 
  \end{cases}
\end{equation}
\end{lemma}
\begin{proof}
Let us denote $\rho_j(t,\xi)=\rho_j(t)$, $\th_j(t,\xi)=\th_j(t)$, $v_0(\xi)=v_0$ and $v_1(\xi)=v_1$. 
We note that 
\begin{equation}
  \begin{pmatrix}
    \rho_1(0)\cos(\th_1(0)) & \rho_2(0)\cos(\th_2(0))\\
    \rho_1(0)\sin(\th_1(0)) & \rho_2(0)\sin(\th_2(0))
  \end{pmatrix}
 =\begin{pmatrix} |\xi|v_0 & 0 \\ 0 & v_1\end{pmatrix}.
\end{equation}
By \eqref{b>=0} and Lemma \ref{lemma_rho-th} with $t_0=0$, we have 
\begin{equation}\label{rhot-rho0}
  \rho_j(t)^2
 =\rho_j(0)^2 \exp\(-2\int^t_{0} b(s)\sin^2(\th_j(s))\,ds\) 
\le \rho_j(0)^2,  
\end{equation}
hence \eqref{est_rhoj} holds for $j=1$ since 
$\rho_1(0) = |\xi|^2 |v_0|^2$. 
Let us consider the case for $j=2$. 
We suppose that $v_1>0$; the other case can also be proved by the same way. 
Then, there exists $t_1>0$ such that 
\begin{equation}
  \sin(\th_2(t)) \ge 0
\end{equation}
on $[0,t_1]$. 
Noting the equality 
\begin{equation}\label{|xi|v-pa_tv}
  |\xi| \rho_j(t)  \sin(\th_j(t)) = \pa_t\(\rho_j(t) \cos(\th_j(t))\)
  \;\; (j=1,2), 
\end{equation}
which follows from \eqref{rho_t-rho} and \eqref{th_t-|xi|}, we have 
\begin{align*}
  0
&\le \int^t_0 |\xi|\rho_2(s)\sin(\th_2(s))\,ds
 =\rho_2(t)\cos(\th_2(t))-\rho_2(0)\cos(\th_2(0))
\\
&=\rho_2(t)\cos(\th_2(t)),
\end{align*}
it follows that 
\begin{equation}\label{cos>0}
  \cos(\th_2(t)) \ge 0 
\end{equation}
on $[0,t_1]$. 
Denoting 
\begin{align*}
  r(t):=
  \sin(\th_2(t)) \exp\(\int^t_0 b(s) \cos^2(\th_2(s))\,ds\),  
\end{align*}
by \eqref{th_t-|xi|} we have 
\begin{align*}
  \pa_t r(t)
&=\cos(\th_2(t))\exp\(\int^t_0 b(s) \cos^2(\th_2(s))\,ds\)
  \(\pa_t \th_2(t) + b(t) \cos(\th_2(t)) \sin(\th_2(t))\)
\\
&=-|\xi|\cos(\th_2(t))\exp\(\int^t_0 b(s) \cos^2(\th_2(s))\,ds\)
\\
&\le 0. 
\end{align*}
It follows that $r(t) \le 1$ on $[0,t_1]$ since $r(0)=1$. 
Therefore, by \eqref{rhot-rho0}, we have 
\begin{align*}
0
&\le \rho_2(t)\sin(\th_2(t))
\\
&=\rho_2(0) \exp\(-\int^t_0 b(s) \sin^2(\th_2(s))\,ds\)
  \sin(\th_2(t)) 
\\
&=\rho_2(0) 
  \exp\(-\int^t_0 b(s)\,ds\)
  \sin(\th_2(t)) 
  \exp\(\int^t_0 b(s) \cos^2(\th_2(s))\,ds\)
\\
&=v_1 \exp\(-\int^t_0 b(s)\,ds\) r(t)
\\
&\le v_1 \exp\(-\int^t_0 b(s)\,ds\)
\\
&\le e^{B_0} v_1 (1+t)^{-m}. 
\end{align*}
Moreover, we have 
\begin{equation}\label{rho2cos<=}
\begin{split}
  0
&\le 
  \rho_2(t) \cos(\th_2(t))
 =\int^t_0 \pa_s\(\rho_2(s) \cos(\th_2(s))\)\,ds
\\
&= \int^t_0 |\xi|\rho_2(s) \sin(\th_2(s)) \,ds
\\
&\le e^{B_0} v_1 |\xi| \int^t_0 (1+s)^{-m} \,ds
\\
&= e^{B_0} v_1 |\xi|\ga(t;m). 
\end{split}
\end{equation}
Therefore, we have 
\begin{align*}
  \rho_2(t)^2
&=\rho_2(t)^2\(\cos^2(\th_2(t)) + \sin^2(\th_2(t))\)
\\
&\le e^{2B_0}|v_1|^2\((1+t)^{-2m} + |\xi|^2  \ga(t;m)^2\)
\end{align*}
on $[0,t_1]$. 
Suppose that $\sin(\th_2(t_1))=0$. 
Then, by \eqref{cE-th} with $t_0=t_1$ 
and \eqref{rho2cos<=} with $t=t_1$, 
for any $t > t_1$ we have 
\begin{align*}
  \rho_2(t)^2
&\le \rho_2(t_1)^2 
  = \rho_2(t_1)^2 \cos^2(\th_2(t_1))
\\
& \le e^{2B_0} |v_1|^2 |\xi|^2 \ga(t_1;m)^2
\\
& \le e^{2B_0} |v_1|^2 |\xi|^2 \ga(t;m)^2. 
\end{align*}
Consequently, we have 
\[
  \rho_2(t)^2 \le e^{2B_0}|v_1|^2\((1+t)^{-2m} + |\xi|^2  \ga(t;m)^2\)
\]
for any $t \ge 0$. 
\end{proof}

\noindent
\textit{Proof of Proposition \ref{Prop1} (ii)}.\;
We note that the solution to \eqref{v} is represented by 
\[
  v(t,\xi) 
 =|\xi|^{-1} 
  \(\rho_1(t,\xi)\cos(\th_1(t,\xi))+\rho_2(t,\xi)\cos(\th_2(t,\xi))\), 
\]
and thus 
\begin{align*}
  \pa_t v(t,\xi)
&=\rho_1(t,\xi)\sin(\th_1(t,\xi))+\rho_2(t,\xi)\sin(\th_2(t,\xi))
\end{align*}
by \eqref{|xi|v-pa_tv}. 
Therefore, by Lemma \ref{lemma_Prop1(ii)}, we have 
\begin{align*}
  \cE(t,\xi)
&\le 2\(\rho_1(t,\xi)^2 + \rho_2(t,\xi)^2\)
\\
&\le 2\(|\xi|^2 |v_0(\xi)|^2
  +e^{2B_0}\((1+t)^{-2m}+|\xi|^2\ga(t;m)^2\)|v_1(\xi)|^2\); 
\end{align*}
thus \eqref{est_cE-ZD1} holds for $K_D \ge 2e^{2B_0}$. 
\qed

%
\subsection{Proof of Proposition \ref{Prop1} (iii)}
%

The following proof of Proposition \ref{Prop1} (iii) below is based on \cite{GH25}. 
In the following proof, we don't use the non-negative condition \eqref{b>=0}, 
but essentially use the non-effective dissipative condition \eqref{0<m<1}. 

Let $v=v(t,\xi)$ be the solution to \eqref{ve}, and consider the following $2\times 2$ matrix $W(t,\xi)=(w_{jk}(t,\xi))$ which defines 
\begin{equation*}\label{vjk}
  \begin{pmatrix}
    |\xi| v(t,\xi) \\ \pa_t v(t,\xi) \\ 
  \end{pmatrix}
 =W(t,\xi)
  \begin{pmatrix}
    |\xi|v_0(\xi) \\ v_1(\xi)
  \end{pmatrix}. 
\end{equation*}
Since the following estimate holds: 
\begin{equation}\label{est-cE_(iii)}
\begin{split}
  \cE(t,\xi;v)
&=\left|w_{11}(t,\xi) |\xi| v_0(\xi) + w_{12}(t,\xi) v_1(\xi)\right|^2
 +\left|w_{21}(t,\xi) |\xi| v_0(\xi) + w_{22}(t,\xi) v_1(\xi)\right|^2
\\
&\le 
  2\(|w_{11}(t,\xi)|^2 + |w_{21}(t,\xi)|^2\)|\xi|^2|v_0(\xi)|^2
\\
&\quad
 +2\(|w_{12}(t,\xi)|^2 + |w_{22}(t,\xi)|^2\) |v_1(\xi)|^2,
\end{split}
\end{equation}
our goal is to derive 
\begin{equation}\label{est-|wjk|}
  \sup_{(t,\xi)\in Z_D}
  \left\{|w_{j1}(t,\xi)| + (1+t)^m |w_{j2}(t,\xi)|\right\}<\infty 
\end{equation}
for $j=1,2$. 

Let us review the following properties of $w_{jk}$: 
\begin{lemma}\label{lamma_vjk}
$w_{1k}(t,\xi)$ and $w_{2k}(t,\xi)$ $(k=1,2)$ are given by the following integral equations: 
\begin{equation}\label{iq_w1k}
   w_{1k}(t,\xi)=
    \de_{1k} + |\xi| \int^t_0  w_{2k}(\tau,\xi)\,d\tau
\end{equation}
and
\begin{equation}\label{iq_w2k}
   w_{2k}(t,\xi)=
    \beta(t)\(\de_{2k} 
  - |\xi| \int^t_0 \beta(\tau)^{-1}  w_{1k}(\tau,\xi)\,d\tau\),
\end{equation}
where 
\[
  \beta(t):=\exp\(-\int^t_0 b(s)\,ds\), 
\]
$\de_{jk}=1$ if $j=k$ and $\de_{jk}=0$ if $j \neq k$. 
\end{lemma}
\begin{proof}
The proof is following straightforward calculations. 
We note that $w_{jk}(0,\xi)=\de_{jk}$. 
Since 
$|\xi|v = w_{11}|\xi|v_0+w_{12}v_1$ and 
$\pa_t v = w_{21}|\xi|v_0+w_{22}v_1$, we have 
\begin{equation}
  |\xi|\pa_t v
 =\begin{cases}
    \pa_t w_{11}|\xi|v_0 + \pa_t w_{12}v_1, \\
    w_{21}|\xi|^2 v_0 + |\xi| w_{22}v_1,
  \end{cases}
\end{equation}
that is, 
\begin{equation}
  \(\pa_t w_{11} -w_{21}|\xi|\)|\xi|v_0
 +\(\pa_t w_{12} -|\xi| w_{22}\)v_1 = 0. 
\end{equation}
It follows that \eqref{iq_w1k} is valid. 
Moreover, since $v$ is a solution to \eqref{ve}, we have 
\begin{align*}
 0&= \pa_t^2 v + |\xi|^2 v + b(t)\pa_t v
\\
  &=\pa_t\(w_{21}|\xi| v_0 + w_{22} v_1\)
   +|\xi|\(w_{11}|\xi| v_0 + w_{12} v_1\)
   +b(t) \(w_{21}|\xi| v_0 + w_{22} |\xi|v_1\)
\\
  &=\(\pa_t w_{21}+b(t)w_{21}+|\xi|w_{11}\)|\xi|v_0
   +\(\pa_t w_{22}+b(t)w_{22}+|\xi|w_{12}\)v_1, 
\end{align*}
it follows that \eqref{iq_w2k} is valid. 
\end{proof}

By substituting \eqref{iq_w2k} into \eqref{iq_w1k}, we have 
\begin{equation}\label{v2k}
\begin{split}
   w_{1k}(t,\xi)
 =\de_{1k} + \de_{2k} |\xi| \int^t_0\beta(s)\,ds
-|\xi|^2 \int^t_0 \(\int^t_{\tau}\beta(s)\,ds \)
  \beta(\tau)^{-1} w_{1k}(\tau,\xi)\,d\tau. 
\end{split}
\end{equation}
Let us define 
$w_1(t,\xi)$, $w_2(t,\xi)$, $p_1(t,\xi)$, $p_2(t,\xi)$ and $q(t,\tau,\xi)$ by 
\[
  w_1(t,\xi) = |\xi| \beta(t)^{-1}  w_{11}(t,\xi),\;\;
  w_2(t,\xi) = \beta(t)^{-1}w_{12}(t,\xi),
\]
\[
  p_1(t,\xi) = |\xi|\beta(t)^{-1},\;\;
  p_2(t,\xi) = |\xi|\beta(t)^{-1} \int^t_0\beta(\tau)\,d\tau
\]
and
\[
  q(t,\tau,\xi)=-|\xi|^2\beta(t)^{-1} \int^t_{\tau}\beta(s)\,ds. 
\]
Then $w_k(t,\xi)$ is a solution to the Volterra integral equation \eqref{Vol} with $p(t)=p_k(t,\xi)$ for $k=1,2$, and we may apply Lemma \ref{lemm-Vol} to obtain \eqref{est-|wjk|}. 

We define $\om(t)$ on $[0,\infty)$ by
\[
  \om(t):=\exp\(\int_t^\infty \frac{\si(\eta(s))}{1+s}\,ds\). 
\]
By Lemma \ref{lemma_int_si0} with $\si_\infty=\si(0)$ and noting that 
$\om(t)=\om(0)(1+t)^m \be(t)$, there exist positive constants $\om_0$ and $\om_1$ such that 
\[
  \om_0:=\inf_{t \ge 0}\{\om(t)\}
  \;\;\text{ \textit{and} }\;\;
  \om_1:=\sup_{t \ge 0}\{\om(t)\}, 
\]
and the following estimates hold: 
\begin{equation}\label{etamu-etab}
  \frac{\om_0}{\om(0)}(1+t)^{-m}
  \le \beta(t) 
  \le \frac{\om_1}{\om(0)}(1+t)^{-m}. 
\end{equation}

\begin{lemma}\label{lemma-est_vjk}
The following estimates hold: 
\begin{equation}\label{lemma-est_vjk-eq}
  |w_{j1}(t,\xi)| \le K_1 
  \;\;\text{ \textit{and} }\;\;
  |w_{j2}(t,\xi)| \le K_2 (1+t)^{-m} 
\end{equation} 
for $j=1,2$ in $Z_D(N)$, where
\[
  K_1 = N\(\frac{\om_1}{\om_0}\)^2 \exp\(\sqrt{\frac{\om_1}{\om_0}}N\)
  \;\;\text{ \textit{and} }\;\;
  K_2 = \frac{\om_1}{\om(0)} \(1+\frac{N K_1}{1-m} \).
\]
\end{lemma}
\begin{proof}
By \eqref{etamu-etab}, we have 
\begin{equation}\label{|p1|}
  |p_1(t,\xi)| \le \frac{\om(0)}{\om_0}|\xi|(1+t)^m, 
\end{equation}
\begin{equation}\label{|p2|}
  |p_2(t,\xi)| 
  \le \frac{\om_1}{\om_0}|\xi|(1+t)^m \int^t_0 \frac{d\tau}{(1+\tau)^{m}}
  \le \frac{\om_1}{\om_0(1-m)}|\xi|(1+t)
\end{equation}
and
\begin{equation}\label{|q|}
  |q(t,\tau,\xi)| 
  \le \frac{\om_1}{\om_0}
  |\xi|^2 (1+t)^m \int^t_{\tau}\frac{d\tau_1}{(1+\tau_1)^{m}}. 
\end{equation}
Suppose that the following inequalities hold for any $l \in \N$: 
\begin{align*}
\int^t_0|q(t,\tau_1,\xi)|\int^{\tau_1}_0|q(\tau_1,\tau_2,\xi)|
  \cdots\int^{\tau_{l-1}}_0|q(\tau_{l-1},\tau_l,\xi)|
  |p_k(\tau_l,\xi)|\,d\tau_l\cdots d\tau_1 
\end{align*}
\begin{equation}\label{intpqj}
\le
\begin{dcases}
  \frac{\om(0)}{\om_0} |\xi|(1+t)^m 
      \frac{1}{(2l)!}\(\sqrt{\frac{\om_1}{\om_0}} |\xi|t \)^{2l} 
 & \textit{for}\;\; k=1,\\
  \frac{\om_1}{\om_0(1-m)}|\xi|(1+t)^m
     \frac{1}{(2l)!}
    \(\sqrt{\frac{\om_1}{\om_0}} |\xi|t \)^{2l} 
 & \textit{for}\;\; k=2. 
\end{dcases}
\end{equation}
Then, by Lemma \ref{lemm-Vol} and $|\xi| t \le N$, 
we have 
\begin{align*}
  |w_1(t,\xi)| 
&\le \frac{\om(0)}{\om_0}|\xi|(1+t)^m 
  + \sum_{j=1}^\infty \frac{\om(0)}{\om_0}|\xi| (1+t)^m
      \frac{1}{(2j)!}\(\sqrt{\frac{\om_1}{\om_0}} N\)^{2j}
\\
&\le \frac{\om(0)}{\om_0}\exp\(\sqrt{\frac{\om_1}{\om_0}} N\)
  |\xi| (1+t)^m. 
\end{align*}
Analogously, we have 
\begin{align*}
  |w_2(t,\xi)| 
 \le
  \frac{\om_1}{\om_0(1-m)}\exp\(\sqrt{\frac{\om_1}{\om_0}} N\)
   |\xi| (1+t).
\end{align*}
Therefore, we have
\begin{align*}
  |w_{11}(t,\xi)|
&\le |\xi|^{-1} \beta(t)|w_1(t,\xi)|
 \le  (1+t)^m \be(t) \frac{\om(0)}{\om_0} \exp\(\sqrt{\frac{\om_1}{\om_0}} N\) 
\\
&\le  \frac{\om_1}{\om_0} 
   \exp\(\sqrt{\frac{\om_1}{\om_0}} N\)
  \le K_1, 
\end{align*}
\begin{align*}
  |w_{21}(t,\xi)|
&\le |\xi| \beta(t) \int^t_0 \beta(\tau)^{-1}|w_{11}(\tau,\xi)|\,d\tau
\\
&\le 
  \(\frac{\om_1}{\om_0}\)^2 \exp\(\sqrt{\frac{\om_1}{\om_0}} N\)
  |\xi| (1+t)^{-m} \int^t_0 (1+\tau)^{m}\,d\tau
\\
&\le\(\frac{\om_1}{\om_0}\)^2 \exp\(\sqrt{\frac{\om_1}{\om_0}} N\)
  |\xi| t
 \le N \(\frac{\om_1}{\om_0}\)^2 \exp\(\sqrt{\frac{\om_1}{\om_0}} N\)
\\
&=K_1
\end{align*}
and
\begin{align*}
  |w_{12}(t,\xi)|
&=\beta(t)|w_2(t,\xi)|
\\
&\le \frac{\om_1}{\om(0)}(1+t)^{-m}
  \frac{\om_1}{\om_0(1-m)}
  \exp\(\sqrt{\frac{\om_1}{\om_0}} N\)
  |\xi|(1+t)
\\
&\le 
  \frac{\om_1^2}{\om_0\om(0)(1-m)} N
  \exp\(\sqrt{\frac{\om_1}{\om_0}} N\)
  (1+t)^{-m}
 = \frac{\om_0 K_1}{\om(0)(1-m)} (1+t)^{-m}
\\
&\le K_2(1+t)^{-m}. 
\end{align*}
Moreover, we have 
\begin{align*}
  |w_{22}(t,\xi)|
&\le \beta(t)
  \(1 + |\xi|\int^t_0 \beta(\tau)^{-1}|w_{12}(\tau,\xi)|\,d\tau\)
\\
&\le \frac{\om_1}{\om(0)}(1+t)^{-m}
  \(1 + \frac{\om(0)}{\om_0}|\xi|\int^t_0 (1+\tau)^m|w_{12}(\tau,\xi)|\,d\tau\)
\\
&\le \frac{\om_1}{\om(0)}(1+t)^{-m} \(1 + \frac{K_1}{1-m} |\xi|t\)
 \le \frac{\om_1}{\om(0)} \(1 + \frac{N K_1}{1-m}\) (1+t)^{-m}
\\
&  =  K_2(1+t)^{-m}. 
\end{align*}
Thus we have \eqref{lemma-est_vjk-eq}. 
Let us prove the inequalities \eqref{intpqj} for $l \ge 2$ by induction. 
By \eqref{|p1|} and \eqref{|q|}, we have 
\begin{align*}
  \int^t_0 |q(t,\tau_1,\xi)| |p_1(\tau_1,\xi)|\,d\tau_1
& \le
  \frac{\om_1 \om(0)}{\om_0^2}
  |\xi|^3 (1+t)^m \int^t_0 (1+\tau_1)^m 
  \int^t_{\tau_1} \frac{d\tau_2}{(1+\tau_2)^{m}} \,d\tau_1
\\
& \le
  \frac{\om_1\om(0)}{\om_0^2}
  |\xi|^3 (1+t)^m \int^t_0 (1+\tau_1)^m 
  \int^t_{\tau_1} \frac{d\tau_2}{(1+\tau_1)^{m}} \,d\tau_1
\\
& =
  \frac{\om_1 \om(0)}{\om_0^2}
  |\xi|^3 (1+t)^{m} \int^t_0 \int^t_{\tau_1} \,d\tau_2 \,d\tau_1
\\
&=
  \frac{\om(0)}{\om_0}
  |\xi| (1+t)^{m} \frac{1}{2!}
  \(\sqrt{\frac{\om_1}{\om_0}}|\xi|t\)^2. 
\end{align*}
Thus \eqref{intpqj} is valid for $k=1$ and $l=1$. 
Suppose that \eqref{intpqj} is valid for $k=1$. 
Then we have 
\begin{align*}
& \int^t_0|q(t,\tau_1,\xi)|\int^{\tau_1}_0|q(\tau_1,\tau_2,\xi)|
  \cdots\int^{\tau_{l}}_0|q(\tau_{l},\tau_{l+1},\xi)|
  |p_1(\tau_{l+1},\xi)|\,d\tau_{l+1}\cdots d\tau_1 
\\
&\le \int^t_0
  \(\frac{\om_1}{\om_0}|\xi|^2 (1+t)^{m}
    \int^t_{\tau_1} \frac{d\tau_2}{(1+\tau_2)^{m}}\)
  \frac{\om(0)}{\om_0} |\xi|(1+\tau_1)^{m}
      \frac{1}{(2l)!}\(\sqrt{\frac{\om_1}{\om_0}} |\xi|\tau_1\)^{2l}
   d\tau_1 
\\
&= \frac{\om(0)}{\om_0}
  \(\sqrt{\frac{\om_1}{\om_0}} |\xi|\)^{2(l+1)}
  |\xi| (1+t)^{m}
  \frac{1}{(2l)!} \int^t_0
  (1+\tau_1)^{m} \int^t_{\tau_1} \frac{d\tau_2}{(1+\tau_2)^{m}}
       \tau_1^{2l}\,
   d\tau_1 
\\
&\le \frac{\om(0)}{\om_0}
  \(\sqrt{\frac{\om_1}{\om_0}} |\xi|\)^{2(l+1)}
  |\xi| (1+t)^{m}
  \frac{1}{(2l)!} \int^t_0
  (1+\tau_1)^{m} \int^t_{\tau_1} \frac{d\tau_2}{(1+\tau_1)^{m}}
       \tau_1^{2l}\,
   d\tau_1 
\\
&\le \frac{\om(0)}{\om_0}
  \(\sqrt{\frac{\om_1}{\om_0}} |\xi|\)^{2(l+1)}
  |\xi| (1+t)^{m}
  \frac{1}{(2l)!} \int^t_0
  \(\int^t_{\tau_1}\,d\tau_2\)
       \tau_1^{2l}\,
   d\tau_1 
\\
&= \frac{\om(0)}{\om_0} |\xi| (1+t)^{m}
  \frac{1}{(2(l+1))!}
  \(\sqrt{\frac{\om_1}{\om_0}} |\xi| t\)^{2(l+1)}. 
\end{align*}
Therefore, \eqref{intpqj} with $k=1$ is valid for any $l \in \N$. 
On the other hand, by the same way as for $k=1$, we have 
\begin{align*}
&  \int^t_0 |q(t,\tau_1,\xi)| |p_2(\tau_1,\xi)|\,d\tau_1
\\
& \le
  \(\frac{\om_1}{\om_0}\)^2
  |\xi|^3 (1+t)^m 
  \int^t_0 
  \(\int^t_{\tau_1}\frac{d\tau_2}{(1+\tau_2)^{m}}\)
  (1+\tau_1)^{m} \int^{\tau_1}_0 \frac{d\tau_2}{(1+\tau_2)^{m}}
  \,d\tau_1
\\
& \le
  \(\frac{\om_1}{\om_0}\)^2
  |\xi|^3 (1+t)^m 
  \int^t_0 
  \(\int^t_{\tau_1}\,d\tau_2\)
  \int^{\tau_1}_0 \frac{d\tau_2}{(1+\tau_2)^{m}}
  \,d\tau_1
\\
& =
  \(\frac{\om_1}{\om_0}\)^2
  |\xi|^3 (1+t)^m
  \(
    \frac12 t^2 \int^{t}_0 \frac{d\tau_1}{(1+\tau_1)^{m}}
   -\int^t_0 \(t\tau_1-\frac12\tau_1^2\) 
    \frac{1}{(1+\tau_1)^{m}} \,d\tau_1
  \)
\\
& \le
  \frac{\om_1}{\om_0} |\xi| (1+t)^m
  \int^{t}_0 \frac{d\tau_1}{(1+\tau_1)^{m}} 
  \frac{1}{2!} 
  \(\sqrt{\frac{\om_1}{\om_0}}|\xi| t\)^2 
\\
& \le \frac{\om_1 |\xi| (1+t)}{\om_0(1-m)} 
  \frac{1}{2!} 
  \(\sqrt{\frac{\om_1}{\om_0}}|\xi| t\)^2. 
\end{align*}
Thus \eqref{intpqj} is valid for $l=1$. 
Suppose that \eqref{intpqj} is valid for $k=2$ and $l \ge 2$. 
Then we have 
\begin{align*}
& \int^t_0|q(t,\tau_1,\xi)|\int^{\tau_1}_0|q(\tau_1,\tau_2,\xi)|
  \cdots\int^{\tau_{l}}_0|q(\tau_{l},\tau_{l+1},\xi)|
  |p_2(\tau_{l+1},\xi)|\,d\tau_{l+1}\cdots d\tau_1 
\\
&\le
   \int^t_0
   \(\frac{\om_1}{\om_0}
   |\xi|^2 (1+t)^m 
   \int^t_{\tau_1} \frac{d\tau_2}{(1+\tau_2)^{m}} \)
   \frac{\om_1}{\om_0(1-m)}|\xi|(1+\tau_1)^m \frac{1}{(2l)!}
   \(\sqrt{\frac{\om_1}{\om_0}} |\xi| \tau_1\)^{2l} \,d\tau_1
\\
&\le
   \frac{\om_1}{\om_0(1-m)} |\xi| (1+t)^m
   \frac{1}{(2l)!} \(\sqrt{\frac{\om_1}{\om_0}} |\xi|\)^{2l+2}
   \int^t_0\(\int^t_{\tau_1} d\tau_2 \) \tau_1^{2l}\,d\tau_1
\\
&= \frac{\om_1}{\om_0(1-m)} |\xi|(1+t)^m
   \frac{1}{(2(l+1))!} \(\sqrt{\frac{\om_1}{\om_0}} |\xi| t\)^{2(l+1)}.
\end{align*}
Therefore, \eqref{intpqj} is also valid for $k=2$ and any $l \in \N$. 
\end{proof}

\noindent 
\textit{Proof of Proposition \ref{Prop1} (iii)}.\;
By \eqref{est-cE_(iii)}, Lemma \ref{lemma-est_vjk} and denoting 
$K_D:=4\max\{K_1^2,K_2^2\}$, we have \eqref{est_cE-ZD2}. 
\qed

%
\subsection{Proof of Theorem \ref{Thm1}}
%
Let $(t,\xi) \in Z_D$. 
If \eqref{b>=0} is valid, then there exists a positive constant $C_m$ such that 
\begin{align*}
  \((1+t)^{-1}\ga(t;m)\)^2
& \le
 \begin{dcases}
    \frac{1}{(m-1)^2}(1+t)^{-2} & (m>1) \\[2mm]
    (1+t)^{-2}\(\log(1+t)\)^2 
      \le 4e^{-2} (1+t)^{-1}& (m=1) \\[2mm]
    \frac{1}{(1-m)^2} (1+t)^{-2m} & (0<m<1)
  \end{dcases}
\\
& \le C_m(1+t)^{-\tm_0}. 
\end{align*}
By Proposition \ref{Prop1} (ii), 
for any $0 < \tm \le \tm_0$ we have 
\begin{align*}
  \cE(t,\xi) 
 &\le K_D 
  \(N^{\tm}(1+t)^{-\tm}|\xi|^{2-\tm} |v_0(\xi)|^2 
  + \((1+t)^{-\tm}+N^2\((1+t)^{-1} \ga(t;m)\)^2\)|v_1(\xi)|^2\)
\\
 &\le \tilde{K}_D(1+t)^{-\tm} \(|\xi|^{2(1-\tm/2)} |v_0(\xi)|^2 
  + |v_1(\xi)|^2\), 
\end{align*}
where $\tilde{K}_D:=K_D \max\{N^{\tm}, 1+N^2 C_m \}$. 
Analogously, if \eqref{0<m<1} is valid, then by Proposition \ref{Prop1} (iii), for any $0 < \tm \le m$ we have 
\begin{align*}
  \cE(t,\xi) 
 &\le 
  K_D (1+t)^{-\tm}
  \((1+t)^{\tm}|\xi|^2 |v_0(\xi)|^2 + (1+t)^{-2m+\tm}|v_1(\xi)|^2\) 
\\
 &\le 
  K_D (1+t)^{-\tm}
  \(N^{\tm} |\xi|^{2(1-\tm/2)} |v_0(\xi)|^2 + (1+t)^{-2m+\tm}|v_1(\xi)|^2\) 
\\
 &\le 
  \tilde{K}_D (1+t)^{-\tm}
  \(|\xi|^{2(1-\tm/2)} |v_0(\xi)|^2 + |v_1(\xi)|^2\). 
\end{align*}
Let $(t,\xi) \in Z_{H1}$. 
By Proposition \ref{Prop1} (i) and the above estimates in $Z_D$, we have 
\begin{align*}
  \cE(t,\xi)
 &\le
  K_H \(\frac{1+t}{1+\cT_0}\)^{-m} \cE(\cT_0,\xi)
  \le
  K_H \(\frac{1+t}{1+\cT_0}\)^{-\tm} \cE(\cT_0,\xi)
  \\
 &\le
  K_H \tilde{K}_D (1+t)^{-\tm}
  \(|\xi|^{2(1-\tm/2)} |v_0(\xi)|^2 + |v_1(\xi)|^2\). 
\end{align*}
If $(t,\xi) \in Z_{H2}$, then we have 
\begin{align*}
  \cE(t,\xi) 
  \le 
  K_H \frac{\exp\(\ve \mu(t)\)}{(1+t)^{m-\tm}} (1+t)^{-\tm} \cE(0,\xi).
\end{align*}
Summarizing the estimates above, we have 
\begin{align*}
  \cE(t,\xi) 
  \le 
  C \(1 + \frac{\exp\(\ve \mu(t)\)}{(1+t)^{m-\tm}}\)
  (1+t)^{-\tm} \( \cE(0,\xi) + |\xi|^{2(1-\tm/2)}|v_0(\xi)|^2\)
\end{align*}
providing \eqref{Eest-Thm}. 
\qed
%
%
%
%
\section{Proof of Theorem \ref{Thm2}}
%
%

Let us introduce the following proposition instead of Proposition \ref{Prop1} (i) in $Z_{H2}$ in order to prove Theorem \ref{Thm2}. 

\begin{proposition}\label{Prop2}
Assume that (A1), (A2) and (A3) and valid. 
For any $\ve>0$, there exist positive constants $K_H$ and $\nu$ such that the following estimate holds in $Z_{H2}$: 
\begin{equation}\label{est_cE-ZH12}
  \cE(t,\xi) 
  \le K_H (1+t)^{-m} \exp\(2\nu \zeta(2|\xi|+\ve)\) \cE(0,\xi). 
\end{equation}

\end{proposition}
\begin{proof}
Let $(t,\xi) \in Z_{H2}$. 
Since $\{\tau_{n+}\}_{n=1}^\infty$ is monotone decreasing, 
in the proof of Lemma \ref{lemma_int-b-th}, 
if $[\tau_-,\tau_+] \subseteq [\tau_{n-},\tau_{n+}]$, then we have 
\begin{align*}
  \left|\int^{\tau_+}_{\tau_-}\frac{g_{j,n}(s)}{1+s}\,ds\right|
  \le 2\ka \mu(\tau_{+})
  \le 2\ka \mu(\tau_{1+})
  \le 2\ka\mu\((\eta')^{-1}\(2|\xi|+\ka\)\)
  = 2\ka \ze\(2|\xi|+\ka\)^2
\end{align*}
for $j=2,4$. 
For $\ve>0$, we define $\ka = \ka(\xi)$ satisfying $0<\ka \le \ve$ by 
\[
  \ka = \ve\frac{\zeta(2N+\ve)}{\zeta(2|\xi|+\ve)}. 
\]
Then \eqref{-2intbsin^2} with $\cT_0=0$ can be replaced as follows: 
\begin{align*}
&  -2\int^t_0 b(s)\sin^2\(\th(s,\xi)\)\,ds
\\
& \quad \le
   \log (1+t)^{-m}
  +B_0
 +\frac{(b_1+2)(2m + A_1)}{N}
 +2\(\frac{2A_1(b_1+2)}{\ve\zeta(2N+\ve)}+\ve A_1\zeta(2N+\ve)\) \zeta(2|\xi|+\ve). 
\end{align*}
Therefore, we have \eqref{est_cE-ZH12} by setting
\[
  K_H = \exp\(B_0 +\frac{(b_1+2)(2m + A_1)}{N}\)
  \;\; \text{ \textit{and} }\;\;
  \nu=\frac{2A_1(b_1+2)}{\ve\zeta(2N+\ve)}+\ve A_1\zeta(2N+\ve).
\]
\end{proof}

%
%
\section{Concluding remarks}
%
%

%
\subsection{Classification of oscillations}
%
Let us review the \textit{oscillations} of the dissipative coefficient that have been considered in this paper. 
Actually, we have not defined the oscillations in the strict sense, but we regarded the properties of the perturbation term $\de(t)$, that is, $\si(\eta)$ and $\eta(t)$, for the monotonic principal term $m/(1+t)$ in \eqref{de(t)} as oscillations. 

For $\si(\eta)$, (A1) can be regarded as an assumption on \textit{smoothness} in the sense of H\"older continuity, and (A2) as an assumption on the \textit{amplitude}. 
Indeed, \eqref{b>=0} means that the amplitude of the sign-changing function $\si(\eta)$ is less than or equal to $m$. 
On the other hand, \eqref{0<m<1} does not explicitly require any restriction on the amplitude of $\si(\eta)$, but it must be bounded by the periodicity and continuity. 
By (A3) and the examples, $\eta(t)$ describes the \textit{frequency} of $\si(\eta(t))$ which is increasing with respect to time. 
Amplitude and frequency, and also smoothness are considered typical properties that describe oscillations, and it ought to be ``quiet'' for the small perturbation. 
It is interesting and may be strange that the oscillations of the perturbation in the frequency sense should not be quiet for the energy decay from the examples of $\eta(t)$. 
However, such an increasing frequency and quietness of the perturbation may coexist if one recall the stabilization condition \eqref{ThmGH-e2}. 

The stabilization condition \eqref{ThmGH-e2} was introduced in \cite{GH25,HW08} as \eqref{ThmGH-e2}. 
In \eqref{ThmGH-e2}, the effect of the perturbation is smaller if $\ga$ can be taken larger. 
For example, if $\de(t)$ is given by 
\begin{equation}\label{de(t)-ex_conc}
  \de(t)=\frac{\sin\((1+t)^\al\)}{1+t}, 
\end{equation}
that is, $\si(\eta)=\sin\eta$ and Example \ref{ex_eta=t} for $\eta(t)$, the maximum choice of $\ga$ for \eqref{ThmGH-e2} is given in \eqref{ThmGH-albe}. 
In this case, since $\ga$ increases as $\al$ increases, the effect of perturbation $\de(t)$ is smaller as the increasing order of the frequency with respect to $t$ is larger, which is consistent with the claim in this paper. 
The effect of $\de(t)$ is also smaller if $\be$ can be taken larger in \eqref{ThmGH-e1}. 
However, the maximum choice of $\be$ given in \eqref{ThmGH-albe} decreases as $\al$ increases; thus larger order of the frequency is not always providing smaller effect of the perturbation. 

%
\subsection{A contribution of initial data with the Gevrey regularity to energy decay estimates}
%

The assertion of Theorem \ref{Thm2} is that even when the oscillations of the perturbation is small in the frequency sense, for which Theorem \ref{Thm1} cannot be applied, it is possible to obtain an energy decay estimate if the initial data is sufficiently smooth such as in the Gevrey class. 
In \cite{GH25}, the following result is proved for the contribution of the initial data in the Gevrey class on the energy decay estimate: 
\begin{theorem}[Example 2.6 \cite{GH25}]\label{ThmGH25-2}
Let $0<m<1$, $\al > 1$, $p < 1$ and $b(t)$ be given by \eqref{de(t)} with
\begin{equation}\label{ex_ThmGH-e1}
  \de(t) = \frac{\sin\((1+t)^\al\)}{(1+t)^p}. 
\end{equation}
If $u_0,u_1 \in G_\infty^s$ with 
\begin{equation}\label{ex_ThmGH-e2}
  1 < s  < \frac{\al-1}{1-p}, 
\end{equation}
then \eqref{Eest-Thm2} holds for $\tm=m$. 
\end{theorem}
%
The oscillations of the perturbation in the frequency sense of the dissipative coefficient $b(t)$ do not have a bad influence on the energy decay estimate according to the consideration of this paper since $\eta(t)=(1+t)^\al$ with $\al>1$. On the other hand, (A2) is not valid since $\de(t)=O(t^{-1})$ ($t\to\infty$) does not hold for $p<1$. 
Thus the oscillations in the amplitude sense of \eqref{ex_ThmGH-e1} are too large to be handled in this paper, but Theorem \ref{ThmGH25-2} asserts that energy decay estimate is possible to be valid due to the contribution of the smoothness of the initial data.

%
\subsection{Wave equations with time dependent propagation speed}
%
There have been many research results on the wave equation with time dependent propagation speed: 
\begin{equation}\label{uW}
  \(\pa_t^2  - a(t)^2 \Delta \) u(t,x) = 0,
\end{equation}
where $a(t)$ satisfy $a_0 \le a(t) \le a_1$ for positive constants $a_0$ and $a_1$. 
One of the most fundamental problem to the initial value problem of \eqref{uW} is to determine the condition to $a(t)$ for the energy estimates 
$C^{-1} E(0;u) \le E(t;u) \le C E(0;u)$ ($C>1$) hold. 
This property is called the generalized energy conservation (GEC) and corresponds to the energy conservation $E(t;u) = E(0;u)$, which holds when $a(t)=1$. 
For example, in \cite{EFH15,H07}, the following conditions were introduced under the assumption $a \in C^k([0,\infty))$ with $k \ge 2$: 
\begin{equation}\label{C^k-w}
  \max_{1 \le l \le k}
  \sup_{t \ge 0}\left\{
  (1+t)^{\be l} \left|a^{(l)}(t)\right|
  \right\}<\infty
\end{equation}
and 
\begin{equation}\label{stb-w}
  \sup_{t \ge 0}\left\{
    (1+t)^{\ga} \int^\infty_{t} |a(s)-a_\infty|\,ds
  \right\}<\infty
  \;\;\text{ \textit{ or } }\;\;
  \sup_{t \ge 0}\left\{
    (1+t)^{\ga} \int^t_{0} |a(s)-a_\infty|\,ds
  \right\}<\infty. 
\end{equation}

By change of time variable $\tau := \int^t_0 a(s)\,ds$, 
\eqref{uW} is reduced to the dissipative wave equation 
\begin{equation}\label{dw_w}
  \(\pa_\tau^2 - \Delta + b(\tau) \pa_\tau\) w(\tau,x) = 0, 
\end{equation}
where 
\[
  \tilde{a}(\tau)=a(t(\tau)),
  \;\;
  b(\tau)=\frac{\tilde{a}'(\tau)}{\tilde{a}(\tau)} 
  \;\;\text{ \textit{and} }\;\;
  w(\tau,x)=w(t(\tau),x). 
\]
Here we note that $a_0 t \le \tau \le a_1 t$ holds. 
GEC for \eqref{uW} and energy decay problem for \eqref{dw_w} are not means equivalent, but are thought to share some common structure. 
Indeed, \eqref{ThmGH-e1} and \eqref{ThmGH-e2} correspond to \eqref{C^k-w} and \eqref{stb-w}, respectively. 
On the problem for GEC, the necessity of the condition $a \in C^2$ has been a matter of concern. 
On the other hand, in the dissipative wave equation \eqref{dw_w}, it corresponds to $b \in C^1$, and \cite{GG25} and our main theorems show that this is not necessarily required if the oscillation in the amplitude sense is small represented by (A2). 
The necessity of $a \in C^3$ is studied in a very recent paper \cite{GGMA}. 
In \cite{GGMA}, it is shown that the conditions for GEC on the wave equation \eqref{uW} can be described by $a \in C^2$ and the stabilization condition \eqref{dw_w}. 
The corresponding problem of the necessity of $b \in C^2$ for the dissipative wave equation \eqref{dw_w} will be considered in the future.

%
%
\section{Appendix}
%
%

\begin{lemma}\label{lemma_int_si0}
Let $\si_0(\eta)$ satisfy (A1) and $\si(\eta)$ be defined by \eqref{si(eta)}. 
Then there exists $\si_\infty \in \R$ such that 
\begin{equation}
  \int^\infty_{0} \frac{\si(\eta(s))}{1+s}\,ds = \si_\infty. 
\end{equation}
\end{lemma}
\begin{remark}
In the following proof, we require only 
$\sum_{n=1}^\infty(|\al_n|+|\be_n|)/n<\infty$ 
for the Fourier coefficient of $\si_0$; 
this is weaker than \eqref{FC}, which requires that the Fourier coefficients satisfy (A1). 
\end{remark}

\begin{proof}
Let us prove the following by restricting to the case $T=\pi$: 
\begin{align*}
  \lim_{T_0,T_1\to\infty}
  \int^{T_1}_{T_0}\frac{\si(\eta(s))}{1+s}\,ds
  =0. 
\end{align*}
For positive real numbers $T_0$ and $T_1$ satisfying $\eta^{-1}(4\pi) \le T_0<T_1$, we define $j_0, j_1 \in \Z$ and $\de_0,\de_1 \in [0,1]$ by 
\[
  j_0=\left[\frac{n\eta(T_0)}{2\pi}\right]+1,
  \;\;
  j_1=\left[\frac{n\eta(T_1)}{2\pi}\right], 
\]
\[
  \de_0
 =j_0-\frac{n\eta(T_0)}{2\pi}
 =1-\(\frac{n\eta(T_0)}{2\pi}-\left[\frac{n\eta(T_0)}{2\pi}\right]\)
\]
and
\[
  \de_1
 =\frac{n\eta(T_1)}{2\pi}-j_1
 =\frac{n\eta(T_1)}{2\pi}-\left[\frac{n\eta(T_1)}{2\pi}\right],
\]
where $[\;\;]$ denotes the floor function. 
Here we note that the following equalites holds: 
\[
  n\eta(T_0)
 =2\pi j_0-2\pi \de_0
  \;\;\text{ \textit{and} }\;\;
  n\eta(T_1)
 =2\pi j_1+2\pi \de_1. 
\]
Denoting 
\[
  \tau=\eta(s)
  \;\;\text{ \textit{and} }\;\;
  \ze(\tau)=\(1 + \eta^{-1}(\tau)\) \eta'\(\eta^{-1}(\tau)\),
\]
it follows that 
\[
  \frac{ds}{d\tau} = \frac{d}{d\tau} \eta^{-1}(\tau)
 = \frac{1}{\eta'(s)} 
 = \frac{1}{\eta'\(\eta^{-1}(\tau)\)}
 = \frac{1+\eta^{-1}(\tau)}{\ze(\tau)}
 = \frac{1+s}{\ze(\tau)}, 
\]
we have 
\begin{align*}
  \int^{T_1}_{T_0} \frac{\sin(n\eta(s))}{1+s}\,ds
&=\int^{\eta(T_1)}_{\eta(T_0)} 
  \frac{\sin(n\tau)}{\ze(\tau)}\,d\tau
 =\int^{n\eta(T_1)}_{n\eta(T_0)} 
  \frac{\sin(s)}{n \ze\(\frac{s}{n}\)}\,ds
\\
&=\int_{2\pi j_0-2\pi \de_0}^{2\pi j_0} 
    \frac{\sin(s)}{n \ze\(\frac{s}{n}\)}\,ds
 +\int_{2\pi j_0}^{2\pi j_1} 
    \frac{\sin(s)}{n \ze\(\frac{s}{n}\)}\,ds
 +\int_{2\pi j_1}^{2\pi j_1+2\pi \de_1} 
  \frac{\sin(s)}{n \ze\(\frac{s}{n}\)}\,ds
\\
&=I_{s,1} + I_{s,2} + I_{s,3}. 
\end{align*}
Since $\ze$ is monotone increasing, we have 
\begin{align*}
  |I_{s,1}|
&\le 
  \frac{1}{n \ze\(\frac{2\pi j_0-2\pi \de_0}{n}\)}
  \left|\int_{2\pi j_0-2\pi \de_0}^{2\pi j_0} \sin(s)\,ds\right|
\\
& =\frac{1}
  {n \ze\(\frac{2\pi}{n}\(\left[\frac{n\eta(T_0)}{2\pi}\right]+1\)
    -\frac{2\pi \de_0}{n}\)}
  \left|\int_{-2\pi \de_0}^{0} \sin(s)\,ds\right|
\\
&\le \frac{2}{n \ze\(\frac{2\pi}{n} \frac{n\eta(T_0)}{2\pi} -2\pi\)}
 =\frac{2}{n \ze\(\eta(T_0) -2\pi\)}, 
\end{align*}
\begin{align*}
  |I_{s,3}|
&\le 
  \frac{1}{n \ze\(\frac{2\pi j_1}{n}\)}
  \left|\int^{2\pi j_1+2\pi \de_1}_{2\pi j_1} \sin(s)\,ds\right|
 =\frac{1}{n \ze\(\frac{2\pi}{n}\left[\frac{n\eta(T_1)}{2\pi}\right]\)}
  \left|\int^{2\pi \de_1}_{0} \sin(s)\,ds\right|
\\
&\le \frac{2}{n \ze\(\frac{2\pi}{n}\(\frac{n\eta(T_1)}{2\pi} -1\)\)}
 \le \frac{2}{n \ze\(\eta(T_1) -2\pi\)}
 \le \frac{2}{n \ze\(\eta(T_0) -2\pi\)}
\end{align*}
and
\begin{align*}
  |I_{s,2}|
&=I_{s,2}
 =\sum_{k=j_0}^{j_1-1}\(
  \int_{2\pi k}^{2\pi k + \pi}
    \frac{\sin(s)}{n\ze\(\frac{s}{n}\)}\,ds
 +\int_{2\pi k + \pi}^{2\pi (k+1)}
    \frac{\sin(s)}{n\ze\(\frac{s}{n}\)}\,ds
  \)
\\
&=\sum_{k=j_0}^{j_1-1}\(
  \frac{2}{n\ze\(\frac{2\pi k}{n}\)}
 -\frac{2}{n\ze\(\frac{2\pi (k+1)}{n}\)}\)
 =\frac{2}{n\ze\(\frac{2\pi j_0}{n}\)}
 -\frac{2}{n\ze\(\frac{2\pi j_1}{n}\)}
\\
&\le \frac{2}{n\ze\(\frac{2\pi j_0}{n}\)}
 =\frac{2}{n\ze\(\frac{2\pi}{n}\(\left[\frac{n\eta(T_0)}{2\pi}\right]+1\)\)}
\\
&\le
  \frac{2}{n\ze\(\frac{2\pi}{n} \frac{n\eta(T_0)}{2\pi}\)}
= \frac{2}{n\ze\(\eta(T_0)\)}. 
\end{align*}
Hence we have 
\begin{align*}
  |I_{s,1}+I_{s,2}+I_{s,3}|
& \le
  \frac{2}{n \ze\(\eta(T_0) -2\pi\)}
 +\frac{2}{n\ze\(\eta(T_0)\)}
 +\frac{2}{n \ze\(\eta(T_0) -2\pi\)}
\\
 &\le
  \frac{6}{n\ze\(\eta(T_0) -4\pi\)}.
\end{align*}
Analogously, denoting
\begin{align*}
  \int^{T_1}_{T_0} \frac{\cos(n\eta(s))}{1+s}\,ds
&=\int_{2\pi j_0-2\pi \de_0}^{2\pi j_0 + \frac{\pi}{2}} 
    \frac{\cos(s)}{n \ze\(\frac{s}{n}\)}\,ds
 +\int_{2\pi j_0+\frac{\pi}{2}}^{2\pi (j_1-1) +\frac{\pi}{2}} 
    \frac{\cos(s)}{n \ze\(\frac{s}{n}\)}\,ds
 +\int_{2\pi (j_1-1) + \frac{\pi}{2}}^{2\pi j_1+2\pi \de_1} 
  \frac{\cos(s)}{n \ze\(\frac{s}{n}\)}\,ds
\\
&=I_{c,1} + I_{c,2} + I_{c,3},
\end{align*}
we have 
\begin{align*}
  |I_{c,1}|
&\le 
  \frac{1}{n \ze\(\frac{2\pi j_0-2\pi \de_0}{n}\)}
  \left|\int_{2\pi j_0-2\pi \de_0}^{2\pi j_0 + \frac{\pi}{2}} 
  \cos(s)\,ds\right|
 \le 
  \frac{2}{n \ze\(\frac{2\pi j_0-2\pi \de_0}{n}\)}
 =\frac{2}{n \ze\(\eta(T_0) -2\pi\)}, 
\end{align*}
\begin{align*}
  |I_{c,3}|
&\le 
  \frac{1}{n \ze\(\frac{2\pi (j_1-1) + \frac{\pi}{2}}{n}\)}
  \left|\int_{2\pi (j_1-1) + \frac{\pi}{2}}^{2\pi j_1+2\pi \de_1} 
  \cos(s)\,ds\right|
 \le \frac{2}{n \ze\(\frac{2\pi (j_1-1) + \frac{\pi}{2}}{n}\)}
\\
&=\frac{2}{n \ze\(\frac{2\pi}{n}\left[\frac{n\eta(T_1)}{2\pi}\right]
  -\frac{3\pi}{2n}\)}
\le \frac{2}{n \ze\(
  \frac{2\pi}{n}\(\frac{n\eta(T_1)}{2\pi}-1\)-\frac{3\pi}{2n}\)}
\\
&= \frac{2}{n \ze\(\eta(T_1)-\frac{7\pi}{2n}\)}
 \le \frac{2}{n \ze\(\eta(T_0)-4\pi\)}
\end{align*}
and
\begin{align*}
  |I_{c,2}|
&=-I_{c,2}
 =\int_{2\pi j_0+\frac{\pi}{2}}^{2\pi (j_1-1) +\frac{\pi}{2}} 
    \frac{-\cos(s)}{n \ze\(\frac{s}{n}\)}\,ds
\\
&=\sum_{k=j_0}^{j_1-2}
  \(
  \int_{2\pi k+\frac{\pi}{2}}^{2\pi k + \frac{3\pi}{2}} 
    \frac{-\cos(s)}{n \ze\(\frac{s}{n}\)}\,ds
 +\int_{2\pi k + \frac{3\pi}{2}}^{2\pi (k+1) +\frac{\pi}{2}} 
    \frac{-\cos(s)}{n \ze\(\frac{s}{n}\)}\,ds
  \)
\\
&\le \sum_{k=j_0}^{j_1-2}
  \(
  \frac{1}{n \ze\(\frac{2\pi k+\frac{\pi}{2}}{n}\)}
  \int_{2\pi k+\frac{\pi}{2}}^{2\pi k + \frac{3\pi}{2}} 
    -\cos(s)\,ds
 +\frac{1}{n \ze\(\frac{2\pi (k+1) +\frac{\pi}{2}}{n}\)}
  \int_{2\pi k + \frac{3\pi}{2}}^{2\pi (k+1) +\frac{\pi}{2}} 
    -\cos(s) \,ds
  \)
\\
&\le \sum_{k=j_0}^{j_1-2}
  \(
  \frac{2}{n \ze\(\frac{2\pi k+\frac{\pi}{2}}{n}\)}
 -\frac{2}{n \ze\(\frac{2\pi (k+1) +\frac{\pi}{2}}{n}\)}
  \)
 =\frac{2}{n \ze\(\frac{2\pi j_0+\frac{\pi}{2}}{n}\)}
 -\frac{2}{n \ze\(\frac{2\pi (j_1-1) +\frac{\pi}{2}}{n}\)}
\\
&\le \frac{2}{n \ze\(\frac{2\pi j_0+\frac{\pi}{2}}{n}\)}
 =\frac{2}{n \ze\(\frac{2\pi}{n}\(\left[\frac{n\eta(T_0)}{2\pi}\right]+1\) + \frac{\pi}{2n}\)}
\\
&\le 
  \frac{2}{n \ze\(\eta(T_0) + \frac{\pi}{2n}\)}
  \le 
  \frac{2}{n \ze\(\eta(T_0)\)}.
\end{align*}
Hence we have 
\begin{align*}
  |I_{c,1}+I_{c,2}+I_{c,3}|
&  \le \frac{1}{n}\(
  \frac{2}{n\ze\(\eta(T_0) -2\pi\)}
 +\frac{2}{n\ze\(\eta(T_0)\)}
 +\frac{2}{n\ze\(\eta(T_0)-4\pi\)}
  \)
\\
&\le \frac{6}{n\ze\(\eta(T_0)-4\pi\)}. 
\end{align*}
Therefore, we have 
\begin{align*}
  \left|\int^{T_1}_{T_0} \frac{\si(\eta(s))}{1+s}\,ds\right|
&\le \sum_{n=1}^\infty
  \(|\al_n|\left|\int_{T_0}^{T_1}\frac{\cos(n\eta(s))}{1+s}\,ds\right|
   +|\be_n|\left|\int_{T_0}^{T_1}\frac{\sin(n\eta(s))}{1+s}\,ds\right|\) 
\\
&\le \frac{6}{\ze\(\eta(T_0)-4\pi\)} \sum_{n=1}^\infty \frac{|\al_n|+|\be_n|}{n}
  \le \frac{6 A_1}{\ze\(\eta(T_0)-4\pi\)}
\\
& \to 0 \;\;(T_0 \to \infty).
\end{align*}
\end{proof}

\begin{lemma}\label{lemm-Vol} 
The solution to the Volterra integral equation 
\begin{equation}\label{Vol}
  w(t) = p(t) + \int^t_0 q(t,\tau) w(\tau)\,d\tau
\end{equation}
is formally represented as follows: 
\begin{align*}
  w(t)
 &=p(t) + \int^t_0 q(t,\tau_1)p(\tau_1)\,d\tau_1
\\
 &\;\; +\sum_{l=2}^\infty 
  \int^t_0 q(t,\tau_1)\int^{\tau_1}_0 q(\tau_1,\tau_2)\cdots
  \int^{\tau_{l-1}}_0 q(\tau_{l-1},\tau_l)p(\tau_l)\,d\tau_l
  \cdots d\tau_1. 
\end{align*}
\end{lemma}
\begin{proof}
The proof can be verified by direct calculations. 
\end{proof}


\end{document}